\documentclass[a4paper,11pt]{article}

\usepackage[utf8]{inputenc}

\usepackage{amsthm}
\usepackage{amsmath}
\usepackage{amssymb}
\usepackage{amsfonts}
\usepackage{dsfont} 

\usepackage{hyperref}
\usepackage{upgreek}
\usepackage{comment}
\usepackage{xfrac}

\usepackage{subfig}

\usepackage{color}

\usepackage{bm}
\usepackage{bbm}
\theoremstyle{plain}
\newtheorem{theorem}{Theorem}[section]

\theoremstyle{definition}
\newtheorem{Def}{Definition}[section]

\theoremstyle{plain}
\newtheorem{prop}{Proposition}[section]

\theoremstyle{plain}
\newtheorem{lemma}{Lemma}[section]

\theoremstyle{plain}

\theoremstyle{remark}
\newtheorem{remark}{{\it Remark}}

\theoremstyle{remark}

\newcommand{\bd}[1]{\begin{Def}\label{#1}}
\newcommand{\ed}{\end{Def}}
\newcommand{\bp}[1]{\begin{prop}\label{#1}}
\newcommand{\ep}{\end{prop}}
\newcommand{\bt}[1]{\begin{theorem}\label{#1}}
\newcommand{\et}{\end{theorem}}

\def\usigma{ \bm{\upvarsigma}}

\newcommand{\eex}{\end{example}}
\newcommand{\de}{\mathrm{d}}

\begin{document}
\large 
\title{Opinion dynamics with Lotka-Volterra type interactions}
\author{Michele Aleandri\footnote{GSSI-Gran Sasso Science Institute, Viale F.~Crispi 17, 67100 L'Aquila, Italy, \url{michele.aleandri@gssi.it}}, Ida G.~Minelli\footnote{Dipartimento di Ingegneria e Scienze
        dell'Informazione e Matematica, Universit\`a degli Studi
        dell'Aquila, Via Vetoio (Coppito 1), 67100 L'Aquila, Italy,
        \url{idagermana.minelli@univaq.it}}}  
\maketitle
\abstract{We investigate a class of models for opinion dynamics in a population with two interacting families of individuals. Each family has an intrinsic mean field ''Voter-like'' dynamics which is influenced by interaction with the other family. 
 The interaction terms describe a  cooperative/conformist or competitive/nonconformist attitude of one family with respect to the other.
We prove chaos propagation, i.e., we show that on any time interval $[0,T]$, as the size of the system goes to infinity, each individual behaves independently of the others with transition rates driven by a macroscopic equation. 
We focus in particular on models with Lotka-Volterra type interactions, i.e., models with cooperative vs. competitive families.  For these models, although the microscopic system is driven a.s. to consensus within each family,  a periodic behaviour arises in the macroscopic scale.\\
In order to describe fluctuations between the limiting periodic orbits, we identify a slow variable in the microscopic system and, through an averaging principle, we find a diffusion which describes the macroscopic dynamics of such variable on a larger time scale.}


\section{Introduction}
A frequent phenomenon observed in social communities is the emergence of self-organized behaviours. In many large communities of randomly interacting individuals, such behaviours appear on a macroscopic scale and seem to follow an independent rule, namely, each individual in the  community feels the influence of other individuals   
through one or more macroscopic variables whose time evolution is deterministic.  
On a first approximation, one can assume that 
 members of a social community are described by identical  
units that evolve randomly in time, choosing their actions from a set of possible ''states'' and  interacting with their ''neighbours''. This assumption has motivated the interest in describing social systems with models based on a statistical physics approach. 
An introduction to the most popular of these models, with a general discussion on the usefulness of ideas and tools of statistical physics in the description of social dynamics, can be found in \cite{castellano2009statistical}. 
Typical questions for these models concern their behaviour when the size of the population or time becomes large.\\ 
Within this context, the field of opinion dynamics models is extremely vast and has attracted researchers from different areas such as social scientists, physicists, computer scientists and mathematicians. All these models differ from one another depending on the set of possible opinions, the structure of the underlying social network and the interaction mechanism between members of the population. 
Without claiming to give a complete description of such a wide field, we limit ourselves to mentioning here few standard examples coming from the classes of discrete and continuous opinion dynamics. According to social scientists (see, e.g., \cite{axelrod1997dissemination}, \cite{flache2011local}, \cite{mcpherson2001birds}), two fundamental characteristics in opinion formation are \emph{social influence}, i.e. the tendency of each individual to adjust her opinion to the one of her neighbours, and \emph{homophily}, i.e. the tendency to 
interact more frequently with individuals who are more similar.   
In dichotomic models, opinions are binary and social influence is usually described in terms of an attractive  interaction between agents. A basic example is the voter model \cite{holley1975ergodic}, where each agent, at random times, adopts the opinion of an agent who is randomly chosen from the set of her neighbours. 
A similar mechanism holds for the Axelrod model (\cite{axelrod1997dissemination}, \cite{lanchier2012axelrod}), where opinions are vector valued (with entries belonging to a finite set) and 
an agent interacts with one neighbour by copying one of the entries of her opinion. 
In continuous dynamics models, opinions are represented by points in a subset of $\mathbb{R}^d$ and each agent may adjust her opinion by adopting a weighted average of her and one (or more) neighbour's opinion.
Examples of such models are the Deffuant-Weisbuch (\cite{deffuant2000mixing}, \cite{haggstrom2014further}) and the Hegselmann-Krause \cite{hegselmann2002opinion}  models.  
In the Axelrod and Deffuant-Weisbuch models, the mechanism of homophily is introduced as follows: two agents interact only if their ''cultural distance'', i.e. the distance between the vectors representing their opinions (which is given by the discrete $L^1$ distance for the Axelrod model and the euclidean distance for the Deffuant-Weisbuch model) does not exceed a certain threshold. 
Models with this feature are known as \emph{bounded confidence models} (see \cite{lorenz2007continuous} for a survey. See also \cite{como2011} for models with heterogeneous populations). 
With this mechanism, convergence to consensus, which typically occurs when social influence is present, may fail  
yielding phenomena such as polarization or fragmentation of opinions within the population. \\ 
A way of describing homophily in dichotomic opinion models could be the introduction of some form of inhomogeneity in the population. For example, one may assume that individuals in the population have different cultural traits, which affect the way one agent's opinion is  influenced by the opinion of other agents (see, e.g., the models considered in \cite{collet2016rhythmic}).\smallskip\\
\indent In this paper we consider a dichotomic opinion model where the population is divided into two social groups, each one characterized by its attitude with respect to the other. 
Members of the same group interact with each other, while the other group exerts on them a social influence, that may also be null or even negative.   
We assume that the cultural characteristics of an individual do not change with her opinions.\\
The model is defined as an interacting particle system with quenched disorder taking values in $\{0,1\}^N$, where $N$ is the size of the population, and can be informally described as follows. 
A population is divided into two families of individuals that may have one of two possible opinions (labelled as $0$ and $1$) on a certain subject.  
For $i=1,2$, an individual of family $i$ chooses at random one member of the population and interaction occurs only if such member belongs to her family:  then, the decision to adopt the opinion of her neighbour is amplified or damped by a perceived utility, which is a (strictly positive) function $\phi_i$ of the fraction of individuals with the same opinion in the other community.\\ The derivative of such functions may be interpreted as a measure of the social influence of one community with respect to the other.  
For example, an increasing function describes a ''cooperative'' attitude, while a decreasing one corresponds to a ''competitive'' attitude. A zealot family may be represented by a $\phi_i$ constant or with a derivative close to zero.  
Other classes of functions can be considered, for example 
the attitude of one family could change from competitive to cooperative if consensus on a given opinion becomes widespread in the other community. 
Notice that this system has four absorbing states, corresponding to configurations where each one of the two families reaches consensus.\\    
We consider the mean field variables  $\bm{m}^N_i=\{\bm{m}^N_i(t)\}_{t\geq 0},\ i=1,2$ where $\bm{m}^N_i(t)$ denotes the fraction of agents with opinion $1$ in family $i$ at time $t$, and we show that they satisfy a law of large numbers: for large $N$, the behaviour of such variables is described by a macroscopic deterministic equation.     
Then we prove chaos propagation, i.e., we show that, for large populations, any finite set of particles evolves as an independent family with jump rates driven by the macroscopic mean field variables.\smallskip\\
 \indent We are mainly interested in the case of a cooperative family interacting with a competitive one. 
For this model, the microscopic interaction between individuals of the two families is a generalization of  the Lotka-Volterra classical interaction, where the utility functions are linear.  
In particular, the macroscopic system evolves through periodic orbits and we are able to identify a quantity $H$ that is  conserved along such orbits.\\  
Stochastic Lotka-Volterra models (see, e.g.,  \cite{klebaner2001asymptotic} and the references therein) have been introduced to study extinction in predator-prey models. 
Indeed, the deterministic models exhibit a cyclic behaviour and extinction is never achieved, while the introduction of noise drives the system towards extinction.   
However, in real social interactions extinction of a given opinion rarely occurs, so we adopt the opposite viewpoint: 
we give a stochastic microscopic description of a bipartite particle system with ''predator-prey'' type interactions. Such system converges a.s. to a configuration where all the members of the same family share the same opinion. On the other hand, letting the size of the population grow to infinity, we obtain a deterministic Lotka-Volterra type dynamics as a result of a law of large numbers.\\
The emergence of orbitally stable solutions in the macroscopic dynamics suggests that the microscopic system spends a considerably long time close to these sets. 
A one dimensional analogue of this scenario is given, for example, in the epidemic model considered in \cite{fagnani2016diffusion}, where the authors show that the macroscopic equation has a stable fixed point close to which the microscopic system spends a time that is  exponential in the size of the population.  
Thus, we consider the microscopic counterpart of the quantity $H$ 
and, through a change of variables, we represent the microscopic system by means of an ''action-angle'' pair $(\bm{H}^N, \bm{\Theta}^N)$ with a slow component $\bm{H}^N=\{\bm{H}^N(t)\}_{t\in [0,T]}$ and a fast one $\bm{\Theta}^N=\{\bm{\Theta}^N(t)\}_{t\in [0,T]}$. 
Then, in order to study how the system fluctuates between the mean field periodic orbits before reaching its absorbing set, we speed up the dynamics and consider the process 
$(\tilde{\bm{H}}^N, \tilde{\bm{\Theta}}^N)= 
(\{\bm{H}^N(Nt)\}_{t\in [0,T]}, \{\bm{\Theta}^N(Nt)\}_{t\in [0,T]})$. 
Following the approach of \cite{DAIPRA2018}, where a  two population Curie-Weiss model is considered, we prove an averaging principle, extending their result to the case when the velocity of the fast variable is not constant. From such principle we derive that, for large $N$, the dynamics of the pair $(\tilde{\bm{H}}^N, \tilde{\bm{\Theta}}^N)$ becomes essentially one dimensional and we prove that the process $\tilde{\bm{H}}^N$ weakly converges, as $N\to \infty$,  to the solution of a stochastic differential equation.\\           
If we interpret as ''more evolved'' a population where two possible opinions coexist and have majorities that change over time, our model suggests that evolution is promoted by cultural diversity, but when the speed of interactions is large compared to the size of the population convergence to consensus within one community is favoured, leading the system to a ''less evolved'' state.  
\\

\section{The model and its mean-field approximation in the quenched regime} \label{section2}
\paragraph{Microscopic system:}

In what follows, we fix two positive real functions, $\phi_1,\phi_2$ of class $\mathcal{C}^2$ on $[0,1]$.
 Consider a filtered  probability space $(\Omega, \mathcal{F}, \{\mathcal{F}_t\},  P)$ satisfying the usual conditions, in which it is defined a family
 $\mathcal{N}=\{ \mathcal{N}^{i,k}; k=1,2, i\geq 1\}$ 
 of i.i.d. adapted Poisson random measures with intensity  
 $\ell \otimes \ell$, where $\ell$ denotes the restriction to $[0, \infty )$
  of the Lebesgue measure, and a probability space $(\Omega^\prime, \mathcal{F}^\prime, \nu)$ in which  
 it is defined a sequence of i.i.d. Bernoulli random variables $\{\Phi_i; i\geq 1\}$ with values in $\{\phi_1, \phi_2\}$ and $\mathbb{P}\{\Phi_i=\phi_1\}=r_1\in (0,1)$.\\  
 Our reference probability space will be $(\bold{\Omega}, \mathcal{A}, \mathbb{P})$, 
 where $\bold{\Omega}=\Omega^\prime\times\Omega,   \mathcal{A}=\mathcal{F}^\prime\otimes\mathcal{F}$ and $\mathbb{P}= \nu\otimes P$.\\ 
 
For a fixed integer $N\geq 2$, we consider  $N$ interacting particles, each one assuming two possible values, 0 or 1. We denote by $\sigma^N$ the particles configuration and by $\sigma^N_{i},\ i=1,\ldots, N$ the state of particle $i$ . 
At each particle $i$ we assign a function $\Phi_i$, which is randomly chosen from $\{\phi_1, \phi_2\}$,  
so that particles are divided into two (random) families, which we call 1 and 2 . 
Let $\mathcal{M}_i=\{h: \Phi_h=\phi_i\}$ be the set of 
neighbours of particle $i$ and $\bar{\mathcal{M}}_i$ its complement in $\{1,\ldots N\}$. Then, conditionally on $\{\Phi_i; i\geq 1\}$,  particle $i$ jumps between states 0 and 1 with the following rates:  
\begin{eqnarray*}
0\to 1\quad &\frac{1}{N}\sum_{j\in\mathcal{M}_i}\sigma_{j}^N\Phi_i(\frac{1}{N}\sum_{j\in \bar{\mathcal{M}}_i}\sigma_{j}^N) ,\\  
1\to 0\quad &
\frac{1}{N}\sum_{j\in\mathcal{M}_i}(1-\sigma_{j}^N)\Phi_i(\frac{1}{N}\sum_{j\in \bar{\mathcal{M}}_i}(1-\sigma_{j}^N))
\end{eqnarray*}
with the convention $\sum_{j\in A}a_j=0$ if $A=\emptyset$. 
Since particles in the same family have the same jump rates, 
denoting, for $k=1,2$, by $N_k=\sum_i I_{\{\Phi_i=\phi_k\}}$ the number of particles in family $k$ and, for $j=1,\ldots N_k$,  by $\sigma^N_{j,k}$ the state of  particle $j$ in family $k$,  for each realization of $\{\Phi_i; i\geq 1\}$,  we can 
write the jump rates for families 1 and 2 as follows:

\begin{eqnarray} \label{rates Nfamily1 01}
0\to 1\quad &\frac{1}{N}\sum_{j=1}^{N_1}\sigma_{j,1}^N\phi_1(\frac{1}{N}\sum_{j=1}^{N_2}\sigma_{j,2}^N) ,\nonumber\\ 
1\to 0\quad &\frac{1}{N}\sum_{j=1}^{N_1}(1-\sigma_{j,1}^N)\phi_1(\frac{1}{N}\sum_{j=1}^{N_2}(1-\sigma_{j,2}^N));\nonumber\\
\ \\
0\to 1\quad &\frac{1}{N}\sum_{j=1}^{N_2}\sigma_{j,2}^N\phi_2(\frac{1}{N}\sum_{j=1}^{N_1}\sigma_{j,1}^N) ,\nonumber\\
1\to 0\quad &\frac{1}{N}\sum_{j=1}^{N_2}(1-\sigma_{j,2}^N)\phi_2(\frac{1}{N}\sum_{j=1}^{N_1}(1-\sigma_{j,1}^N)).\nonumber
\end{eqnarray}

We consider the system in the quenched regime, so that we have $\nu-$ a.s.  
\begin{equation*}
\frac{N_k}{N}\longrightarrow r_k,\quad k\in\{1,2\}
\end{equation*} 
where $r_2=1-r_1$. Moreover, since we want to study the system for large $N$, we can assume without loss of generality that $N_1, N_2>0$ for all $N$.\\

Let $m_1^N,m_2^N$ be the fraction of 1's of the first and second family, i.e.,  $m_k^N=\frac{1}{N_k}\sum_{j=1}^{N_k}\sigma_{j,k}^N$, $k=1,2$. We can rewrite the rates \eqref{rates Nfamily1 01} as:
\begin{equation*}
c(i,k, \sigma^N)= \left\{ \begin{array}{ll}
\frac{N_k}{N}m_k^N\phi_k(\frac{N_{k'}}{N}m_{k'}^N)  & \mbox{if }\  \sigma^N_{i,k}=0\\  
\frac{N_k}{N}(1-m_{k}^N)\phi_k(\frac{N_{k'}}{N}(1-m_{k'}^N)) & \mbox{if }\  \sigma^N_{i,k}=1
\end{array}
\right.
 \end{equation*}
where $k'=3-k$ and $k\in\{1,2\}$, so that  
the $N$-particles system in the quenched regime is described by the Markov process on $\{0,1\}^N$ with generator:  
\begin{equation}\label{generatorN}
\mathcal{L}_Nf (\sigma^N)=\sum_{k=1}^2\sum_{i=1}^{N_k} c(i,k, \sigma^N)[f(\sigma^{N,i,k})-f(\sigma^N)]
\end{equation}
where $f:\{0,1\}^N\to\mathbb{R}$ and $\sigma^{N,i,k}$ denotes the configuration obtained by $\sigma^N$ by replacing $\sigma^N_{i,k}$ 
with $1-\sigma^N_{i,k}$. Note that the process has four  absorbing states, corresponding to the configurations where all the particles within a given family have the same state. In the language of opinion dynamics, such configurations are usually called ''consensus'', when all the particles in the population share the same state, or ''polarization'' otherwise.\\ 

In what follows, we shall use the bold notation 
$\usigma^N=\{\usigma^N(t)\}_{t\geq 0}$ to denote a Markov process with generator (\ref{generatorN}). We denote  
by $\usigma^N_{i,k}(t)$  the state of particle $i$ of family $k$ at time $t$ and by $ \big(\bm{m}^N_1, \bm{m}^N_2\big)$ the stochastic process $\{(\bm{m}^N_1(t), \bm{m}^N_2(t))\}_{t\geq 0}$ where  $\bm{m}_k^N(t)=\frac{1}{N_k}\sum_{j=1}^{N_k}\usigma_{j,k}^N(t)$, $k=1,2$.\\ 
The process $\usigma^N$ can be realized on $(\Omega, \mathcal{F}, \{\mathcal{F}_t\}_t, P)$ as the solution of the following SDE: 
\begin{equation}\label{equation micro poisson}
 \de \bm{\upvarsigma}_{i,k}^N(t)= \int_0^\infty\psi( \bm{\upvarsigma}_{i,k}^N(t-))\mathbbm{1}_{(0,\lambda^N( \bm{\upvarsigma}_{i,k}^N(t-),\bm{m}_k^N(t-),\bm{m}_{k'}^N(t-))]}(u)\mathcal{N}^{i,k}(\de u,\de t)
\end{equation}
$i=1,\ldots , N$, $k=1,2$, $k'=3-k$, 
where  $\lambda^N:\{0,1\}\times[0,1]\times[0,1]\to \mathbb{R}_+$ is  the jump rate function 
\begin{equation}\label{rate function lambdaN}
\lambda^N(\sigma_{i,k}^N,m_k^N,m_{k'}^N)= (1-\sigma_{i,k}^N)\frac{N_k}{N}m_k^N\phi_k\left(\frac{N_{k'}	}{N}m_{k'}^N\right) + \sigma_{i,k}^N\frac{N_k}{N}(1-m_k^N)\phi_k\left(\frac{N_{k'}	}{N}(1-m_{k'}^N)\right)
\end{equation}
and  $\psi:\{0,1\}\to\mathbb{R}_+$ is the jump amplitude function
\begin{equation*}
\psi(\sigma_{i,k}^N)=1-2\sigma_{i,k}^N.
\end{equation*}

\begin{remark}
Note that the functions $\psi$ and $\lambda^N$ are uniformly bounded. Moreover, since $\psi$ and $\phi_k$ are Lipschitz functions, if we pose  
$f(\sigma^N_{i,k}, u)=\psi (\sigma^N_{i,k})\mathbbm{1}_{(0,\lambda^N( \sigma_{i,k}^N, m_k(\sigma^N), m_{k'}(\sigma^N))]}(u)$, where $m_k(\sigma^N)=\frac{1}{N_k}\sum_{j=1}^{N_k}\sigma^N_{j,k}$, the following Lipschitz condition holds:  
$$\sum_{i,k}\int |f(\sigma^N_{i,k}, u)-f(\eta^N_{i,k}, u)| \de 	u\leq C  
 \|\sigma^N-\eta^N\| \ \ \ \mbox{for all } \sigma^N, \eta^N\in \{0,1\}^N,$$
where $\|\cdot \|$ denotes the $L^1$ norm on $\mathbb{R}^N$ and $C$ is a suitable constant.  
Strong existence and uniqueness of solutions to (\ref{equation micro poisson}), with any initial condition $\usigma^N(0)$ independent of $\mathcal{N}$, can be derived by adapting the proof of Theorem 1.2 in \cite{graham1992mckean}.

 \end{remark}
 
We then obtain a family of Markov processes $\big\{\usigma^N\ ; N>1\big\}$ where $\usigma^N$ has sample paths in the space of càdlàg  functions
$ \mathcal{D}([0,\infty), \mathbb{R}^N)$ and generator given by $\mathcal{L}_N$.

 \paragraph{Macroscopic system:}
At an heuristic level, let us make the assumption that a law of large numbers holds for $\big\{\bm{m}^N_k\ ; N>1\big\}$, $k=1,2$,  i.e., it converges, as $N\to\infty$, to a deterministic function $m_k$. Then, for large $N$, the system can be  described by a macroscopic dynamics: if  $\bm{\upvarsigma}_k$ denotes the state of the ''limit particle'' of family $k$, we expect that it evolves as a time-inhomogeneous Markov process with jump rates:
\begin{equation}\label{macrosigma}
\begin{array}{ll}
0\to 1\quad & r_km_k\phi_k(r_{k'}m_{k'}),\\
1\to 0\quad &r_k(1-m_k)\phi_k(r_{k'}(1-m_{k'}))
\end{array}
\end{equation}
where $m_k(t)=E[\usigma_k(t)]$ for all $t$ and $m_k$ satisfy 
a suitable evolution equation.\\ 
We can obtain such equation using the generator of the Markov process above, which we denote by $\mathcal{L}$:
\begin{multline}\nonumber
\dot{m}_k=\frac{d}{dt}E[\bm{\upvarsigma}_k]=E[\mathcal{L}\bm{\upvarsigma}_k] \\
= E[ (1-\bm{\upvarsigma}_k)r_km_k\phi_k(r_{k'}m_{k'}) - \bm{\upvarsigma}_kr_k(1-m_k)\phi_k(r_{k'}(1-m_{k'}))] \\
= m_k(1-m_k)r_k\big[\phi_k(r_{k'}m_{k'})-\phi_k(r_{k'}(1-m_{k'}))\big].
\end{multline}
We obtain the system:
\begin{equation}\label{macro}
\begin{array}{ll}
\dot{m}_1=r_1m_1(1-m_1)[\phi_1(r_2m_2)-\phi_1(r_2(1-m_2) )];\\
\dot{m}_2=r_2m_2(1-m_2)[\phi_2(r_1m_1)-\phi_2(r_1(1-m_1) ) ]. 
\end{array}
\end{equation} 
The following proposition shows indeed that, as $N\to +\infty$, the sequence $\{(\bm{m}_1^N, \bm{m}_2^N) ; N>1\}$ converges in distribution to the deterministic process $(m_1, m_2)$ described by equation (\ref{macro}). 

\begin{prop}	
Suppose there exists a non-random pair $(\bar{m}_1,\bar{m}_2)\in[0,1]^2$ such that, for every $\epsilon>0$, 
\begin{equation*}
\lim_{N\to+\infty}\max_{k=1,2}P\big(|\bm{m}_k^N(0)-\bar{m}_k|>\epsilon\big)=0.
\end{equation*} Then the sequence of Markov processes $\big\{ (\bm{m}_1^N,\bm{m}_2^N)\ ; N>1\big\}$ converges in distribution, as $N\to +\infty$, to the unique solution of  equation (\ref{macro}) with $(m_1(0), m_2(0))=(\bar{m}_1, \bar{m}_2)$.
\end{prop}  
\begin{proof}
Let  $\mathcal{L}_N$ be the generator of the evolution of the particles and $E^N=\{(x_1,x_2)\in [0,1]^2: x_k=\frac{j}{N_k},  0\leq j\leq N_k, k=1,2 \}$. For $f:E^N\to\mathbb{R}$ we can write 
$f(m_1^N, m_2^N)=(f\circ h) (\sigma^N)$ for a suitable function $h:\{-1,1\}^N\to E^N$. Then, a direct computation yields:
\begin{equation*}
\mathcal{L}_N(f\circ h)(\sigma^N)=\mathcal{G}_N f(m_1^N,m_2^N)
\end{equation*}
where
\begin{eqnarray*}
\mathcal{G}_Nf(x,y)&=& N_1 \frac{N_1}{N}x(1-x)\phi_1\left ( 
\frac{N_{2}}{N}y\right ) \left [ f\left (x+\frac{1}{N_1},y\right )-f(x,y)
\right ] \\
&+&  N_1\frac{N_1}{N} x(1-x)\phi_1\left ( \frac{N_{2}}{N}(1-y)\right ) 
 \left [ f
\left (x-\frac{1}{N_1},y\right )-f(x,y)\right ] \\
&+&  N_2 \frac{N_2}{N}y(1-y)\phi_2\left (\frac{N_{1}}{N}x\right ) \left [ f\left 
 (x,y+\frac{1}{N_2}\right )-f(x,y)\right ]\\
 &+& N_2 \frac{N_2}{N}y(1-y)\phi_2\left (\frac{N_{1}}{N}(1-x)\right )\left [ f  \left (x,y-\frac{1}{N_2}\right )-f(x,y)\right ] .  
\end{eqnarray*}
Denote by $\mathcal{G}$ the generator of the semigroup associated to the deterministic evolution (\ref{macro}). If $f\in\mathcal{C}^1([0,1]^2)$, one checks that: 
\begin{equation*}
\lim_{N\to+\infty}\sup_{(x,y)\in E^N}|\mathcal{G}_Nf(x,y)-\mathcal{G}f(x,y)|=0.
\end{equation*}
The conclusion then follows applying standard results on convergence of Markov processes (see, e.g., \cite{EthierKurtz}, Ch.\ 3, Corollary 7.4 and Ch.\ 4, Theorem 8.10).
\end{proof}

\noindent Figures 1 and 2 below show a picture of the solutions of the macroscopic equation \eqref{macro} for different choices of $\phi_i,\ i=1,2$.\smallskip\\
Now we analyse equation \eqref{macro} in the case when $\phi_1$ and $\phi_2$ are \emph{strictly monotonic} functions.  
We define,  for $i,j=1,2, i\neq j$ and $z\in [0,1]$: 
\begin{equation*}
\psi_i(z):=\phi_i(r_j z)-\phi_i(r_j(1-z)).
\end{equation*}
Since, for all $z$,   
 $\psi_i(z)=\phi_i^\prime(r_j\xi)r_j(2z-1)$ for some convex combination $\xi$ of $z$ and $1-z$, the set of fixed points 
of (\ref{macro}) is given by $S=\{(0,0), (0,1), (1,0), (1,1), (\frac{1}{2},\frac{1}{2})\}$; their stability, as can be easily checked by linearising the system, depends on the sign of $\psi_i(0), i=1,2$, which in turn depends 
on the sign of $\phi_i^\prime, i=1,2$. 
%
%
In particular, when both $\phi_1$ and $\phi_2$ are increasing (resp. decreasing), the points $(0,0)$ and $(1,1)$ (resp. $(0,1)$ and $(1,0)$) are stable, while   
 $(0,1)$ and $(1,0)$ (resp. $(0,0)$ and $(1,1)$) are unstable. 
Moreover, 
the characteristic equation for the Jacobian matrix at $(1/2, 1/2)$  is given by: 
\begin{equation*}
\lambda^2- \frac{(r_1r_2)^2}{4}\phi_1^\prime(\frac{r_2}{2})\phi_2^\prime(\frac{r_1}{2})=0 
\end{equation*}
and, 
for $\phi_1^\prime \phi_2^\prime>0$, the point $(1/2, 1/2)$ is unstable.\\
Now, let us assume that $\phi_1$ is increasing  and $\phi_2$ is decreasing. In this case, the point  $(1/2, 1/2)$ is a center for the linearised system. 
Indeed, consider equation (\ref{macro}) for $(m_1,m_2)\in (0,1)^2$; multiplying both terms of the first equation by 
$\frac{\psi_2(m_1)}{r_1m_1(1-m_1)}$ and using the second equation we obtain:  
\begin{equation*}
\frac{\psi_2(m_1)}{r_1 m_1(1-m_1)}\dot{m}_1-\frac{\psi_1(m_2)}{r_2 m_2(1-m_2)}\dot{m}_2=0. 
\end{equation*}
 Then, if we pose  $\Psi_1(z):=\int \frac{\psi_2(z)}{r_1z(1-z)}\de z$ and 
 $\Psi_2(w):=-\int \frac{\psi_1(w)}{r_2w(1-w)}\de w$, the function 
 $H: (0,1)^2\to\mathbb{R}$ defined by $H(z,w)=\Psi_1(z)+\Psi_2(w)$ is such that 
 $\frac{dH}{dt}(m_1, m_2)=0$, and so the sets of the form 
 $\mathcal{C}_k=\{(z,w)\in (0,1)^2: H(z,w)=k\}$ are invariant for the dynamics (\ref{macro}). \\ 
The Hessian of $H$ is diagonal with entries given by 
 $h_{11}= \frac{\psi^\prime_2(z)}{r_1z(1-z)}-\psi_2(z)\frac{2z-1}{r_1z^2(1-z)^2}$   
 and $h_{22}= -\frac{\psi^\prime_1(w)}{r_2w(1-w)}+\psi_1(w)\frac{2w-1}{r_2w^2(1-w)^2}$.
Recalling that $\psi_2^\prime<0$, $\psi_1^\prime >0$ and, for $i=1,2$,  $\ \psi_i(z)=\psi_i^\prime(\xi_z)(2z-1)$ where $\xi_z=\alpha z +(1-\alpha)(1-z)$ for some $\alpha\in (0,1)$, we have $h_{ii}<0$ for $i=1,2$, so that $H$ is a concave function with maximum at 
$(1/2, 1/2)$. 
Notice that, for all $\bar{z}\in (0,1)$,  $\Psi_1(\bar{z})-\Psi_1(\frac{1}{2})=\int_{1/2}^{\bar{z}} \frac{\psi_2(z)}{r_1z(1-z)}\de z=
\int_{1/2}^{\bar{z}} \frac{-\psi_2(1-z)}{r_1z(1-z)}\de z= 
\int_{1/2}^{1-\bar{z}}\frac{\psi_2(w)}{r_1w(1-w)}\de w$ from which it follows $\Psi_1(z)=\Psi_1(1-z)$. Analogously, $\Psi_2(w)=\Psi_2(1-w)$, then 
for $(z,w)\in \mathcal{C}_k$ we have also $(1-z, w), (z, 1-w), (1-z, 1-w)\in \mathcal{C}_k$ and this shows that  $\mathcal{C}_k$ is a closed curve. Moreover, one can easily check that the curves $\mathcal{C}_k$ are orbitally stable solutions of \eqref{macro} and that,  
as $(z,w)$ approaches the boundary of the square $[0,1]^2$, the Hamiltonian $H$ tends to $-\infty$.
\begin{remark}
The above computation shows that the system undergoes a Hopf bifurcation determined by the sign of $\phi'_1(\frac{r_2}{2})\phi_2'(\frac{r_1}{2})$. 
When $\phi_1$ and $\phi_2$ are not both montone, other equilibria may appear in the macroscopic equation and their stability depends locally on the sign of $\phi_i',\  i=1,2$. In particular, 
periodic orbits may be observed around different points of the phase space (see figure 2).   
\end{remark}

\begin{figure}[t]
\centering
\subfloat[][$\phi_2'>\phi'_1>0$]
{\includegraphics[height=%
0.25\textheight]{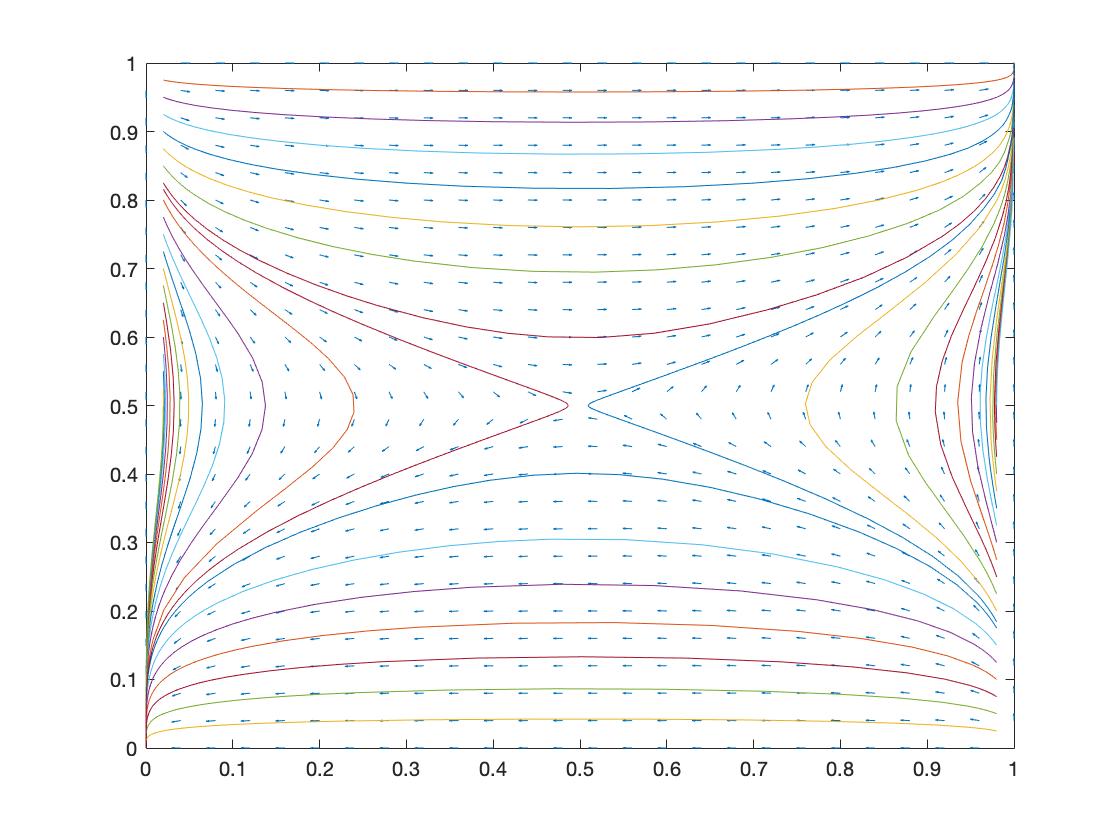}}
\subfloat[][$\phi'_1>0, \phi_2'<0$]
{\includegraphics[height=%
0.25\textheight]{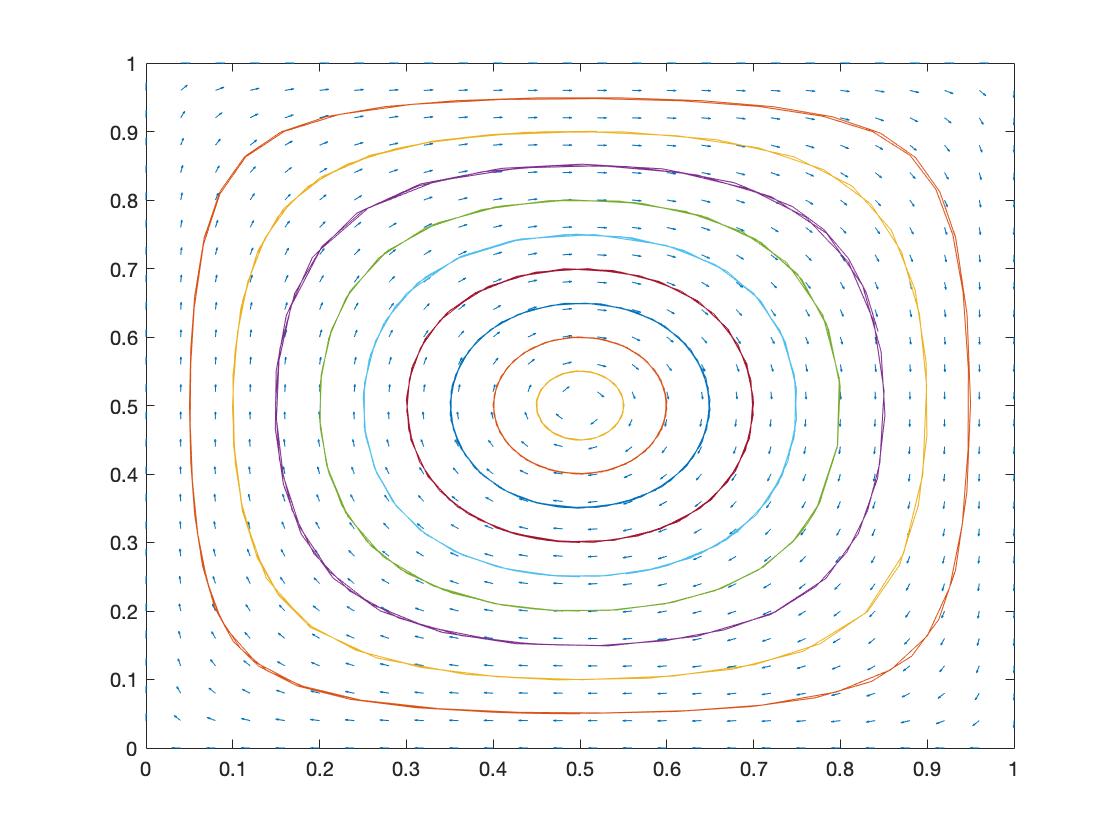}}
\caption{Trajectories of the macroscopic system for  $\phi_1$, $\phi_2$ monotonic.}
\end{figure}

\begin{figure}[t]
\centering
\subfloat[][$\phi_1=\phi_2$]
{\includegraphics[height=%
0.25\textheight]{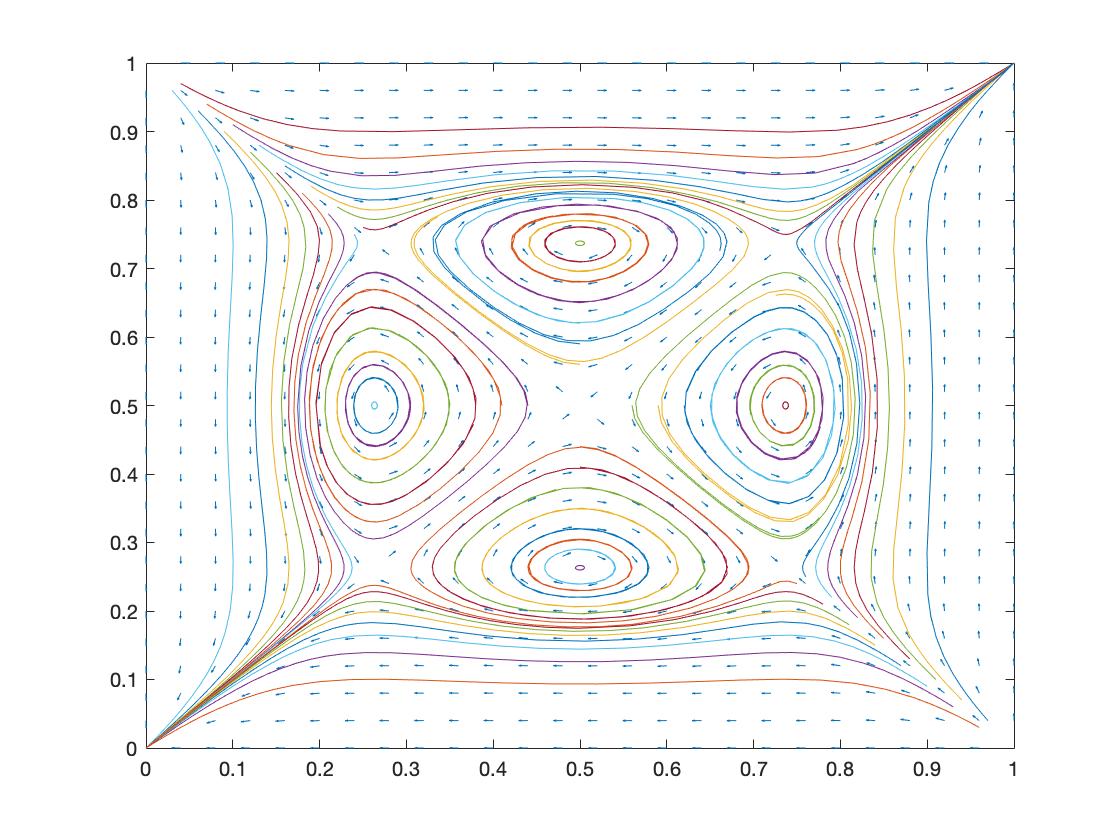}}
\subfloat[][$\phi_1=-\phi_2$]
{\includegraphics[height=%
0.25\textheight]{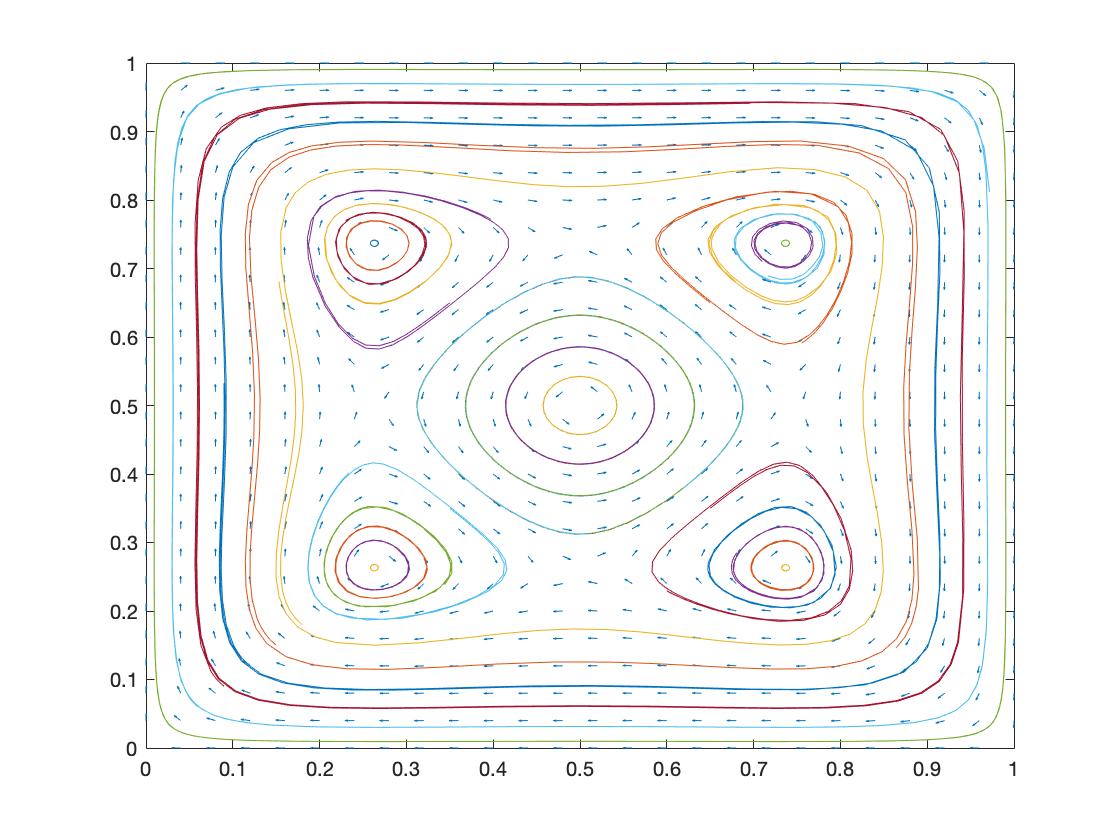}}
\caption{Trajectories of the macroscopic system for $\phi_1$, $\phi_2$ non monotonic.}
\end{figure}


\section{Propagation of chaos}
In this section we prove \emph{propagation of chaos}, i.e., as $N\to \infty$ particles of both families behave independently according to the evolution  (\ref{macrosigma}) with transition rates depending on the 
solution $(m_1, m_2)$ of equation (\ref{macro}). 
To this purpose, we will use a coupling  technique following the approach of \cite{graham1992mckean} and \cite{Andreis}.

\begin{Def}
Let $(E,d)$ be a Polish space, $\mu$ a probability measure on $E$ and, for each $N\geq 1$, let 
$\mu^N$ be a probability measure on $E^N$. 
For a fixed integer $n$, denote by $\mu_{1,\ldots , n}^N$ the marginal distribution of $\mu^N$ over the first $n$ components. 
The sequence $\{\mu^N; N\geq 1\}$ is said to be $\mu$\emph{-chaotic} if,  for each $N$, $\ \mu^N$ is permutation invariant and
for every $n<N$  the sequence $\{\mu_{1,\ldots , n}^N; N\geq 1\}$ converges weakly to  the product measure $\mu^{\otimes n}$ as $N\to \infty$.\\ 
We say that \emph{propagation of chaos holds} for a sequence of  random vectors $\{X^N; N\geq 1\}$, where $X^N$ takes values on $E^N$, if the sequence of their distributions is $\mu$-chaotic for some probability $\mu$ on $E$.  
\end{Def}
A stronger notion of chaoticity uses convergence with respect to the Wasserstein distance, which implies weak convergence. 
Let $\mathcal{M}_1(E)$ be the set of probability measures on $E$ with finite first moment. The Wasserstein metric on $\mathcal{M}_1(E)$ is defined by: 
$$W_d^1(\mu,\nu)= \inf\left\{\int d(x,y) \pi(\de x, \de y): \pi \mbox{ has marginals } \mu \mbox{ and } \nu\right\}.$$ 
For $n\geq 1$ and $T>0$, we call $\rho_n$ and $\rho_{n,T}$ the Wasserstein distances $W^1_{\| \cdot \|}$ on $\mathcal{M}_1(\mathbb{R}^n) $  and $W_{\| \cdot \|_{\infty}}^1$ on $\mathcal{M}_1(\mathcal{D}([0,T]; \mathbb{R}^n))$ respectively, where $\| \cdot \|$ denotes the $L^1$ metric on $\mathbb{R}^n$ and $\| \cdot \|_{\infty}$ denotes the uniform metric on the Skorohod space of càdlàg functions  $\mathcal{D}([0,T]; \mathbb{R}^n)$. 

\begin{Def}
 Let $\mu\in \mathcal{M}_1(\mathbb{R})$ (respectively,  $\mu\in \mathcal{M}_1(\mathcal{D}([0,T]; \mathbb{R}))$). 
We say that a sequence $\{\mu^N; N\geq 1\}$ of permutation invariant probability measures  (or, equivalently, a sequence of random vectors $\{X^N; N\geq 1\}$, where $X^N$ has distribution $\mu^N$) on $\mathbb{R}^N$
 (resp. $\mathcal{D}([0,T]; \mathbb{R}^N)$)
 is  $\mu$\emph{-chaotic in} $W^1$, if, for each $n\geq 1$, the sequence $\{\mu^N_{1,\ldots, n}\}$ converges to $\mu^{\otimes n}$ with respect to the metric $\rho_n$ (resp. $\rho_{n, T}$). 
\end{Def}
Now, consider the following SDE on $(\Omega, \mathcal{F}, \{\mathcal{F}_t\}_t, P)$: 

\begin{equation}\label{equation macro poisson}
\de \bar{\bm{\upvarsigma}}_{i,k}^N(t)= \int_0^\infty\psi(\bar{\bm{\upvarsigma}}_{i,k}^N(t-))\mathbbm{1}_{(0,\lambda (\bar{\bm{\upvarsigma}}_{i,k}^N(t-), m_k(t-), m_{k'}(t-))]}(u)\mathcal{N}^{i,k}(\de u,\de t)
\end{equation}

with $i=1,\ldots , N$, $k=1,2$, $k'=3-k$, 
where  $(m_1, m_2)$ is the solution of equation \eqref{macro}, the Poisson random measures  $\mathcal{N}^{i,k},  i=1,\ldots , N_k, k=1,2$ are the same of equation \eqref{equation micro poisson}, the jump rate function is given by
\begin{equation*}
\lambda(\sigma_{i,k},m_k,m_{k'})= (1-\sigma_{i,k})r_km_k\phi_k(r_{k'}m_{k'}) +  \sigma_{i,k}r_k(1-m_k)\phi_k(r_{k'}(1-m_{k'}))
\end{equation*}
and the jump amplitude is 
\begin{equation*}
\psi(\sigma_{i,k})=1-2\sigma_{i,k}.
\end{equation*}

The solution $\bar{\usigma}^N$ of equation \eqref{equation macro poisson} is given by a system of $N=N_1+N_2$ particles evolving independently on $\{0,1\}$ with jump rates \eqref{macrosigma} and it is coupled with the solution $\usigma^N$ of equation \eqref{equation micro poisson} through the random measures $(\mathcal{N}^{i,k})$.  
Such a coupling allows to prove propagation of chaos for the sequence $\{\usigma^N; N>1\}$. 

\begin{prop}\label{prop chaos}
For $N>1$, let $\{\usigma^N_k=(\usigma^N_{i,k})_{1\leq i\leq N_k}; k=1,2\}$ and  $\{\bar{\usigma}^N_k=(\bar{\usigma}_{i,k}^N)_{1\leq i\leq N_k}; k=1,2\}$ be the solutions of the microscopic equation \eqref{equation micro poisson} and  the macroscopic equation \eqref{equation macro poisson} respectively, with initial conditions $\usigma^N_k(0)$ and $\bar{\usigma}^N_k(0),\ k=1,2$, independent of the family of Poisson random measures $\mathcal{N}$.  Denote by $\mu_{[0,T]}^{(k)}$ the law of $\{\bar{\usigma}^N_{1,k}(t)\}_{t\in [0,T]}$. \\ 
Assume that, for $k=1,2$,  $(\bar{\usigma}_{i,k}^N(0))_{1\leq i\leq N_k}$ are i.i.d. with common distribution $\mu_{0}^{(k)}$ on $\{0,1\}$, 
$\{\usigma^N_k(0); N>1\}$ is $\mu_{0}^{(k)}$-chaotic in $W^1$ and  
$\lim_{N\to\infty}E\big[ |\usigma_{i,k}^N(0)-\bar{\usigma}_{i,k}^N(0)|\big] =0$. Then, for $k=1,2$ and for any $T>0$, the sequence $\{\{\usigma^N_k(t)\}_{t\in [0,T]}; N>1\}$ is $\mu_{[0,T]}^{(k)}$-chaotic in $W^1$. 
\end{prop}
 
\begin{proof}
Let us fix $k\in \{1,2\}$. Clearly, for each $N>1$, the distribution of $\{\usigma^N_k(t)\}_{t\in [0,T]}=\{(\usigma^N_{i,k}(t))_{1\leq i\leq N_k}\}_{t\in [0,T]}$ is permutation invariant. 
To prove chaoticity in $W^1$ it is enough to prove that, for any $T>0$ and $i\in \{1,\ldots, N_k\}$ we have:
\begin{equation} \label{convergence of mean value}
E\Big[ \sup_{t\in[0,T]}\left|\usigma_{i,k}^N(t)-\bar{\usigma}_{i,k}^N(t)\right|\Big]\xrightarrow{N\to\infty} 0.
\end{equation}

For a shorter notation we write:
\begin{eqnarray*}
\bm{\lambda}^N_{i,k}(s-)&:=&\lambda^N(\usigma_{i,k}^N(s^-),\bm{m}_k^N(s^-), \bm{m}_{k'}^N(s^-)),\\
\bar{\bm{\lambda}}_{i,k}(s-)&:=& \lambda (\bar{\usigma}_{i,k}^N(s^-), m_k(s^-), m_{k'}(s^-)),\\
\Lambda^N_{i,k}&:=& E\big[ |\usigma_{i,k}^N(0)-\bar{\usigma}_{i,k}^N(0)|\big].
\end{eqnarray*}

For any $t\geq 0$ we have:   
\begin{eqnarray*}
& &\sup_{r\in[0,t]}|\usigma_{i,k}^N(r)-\bar{\usigma}_{i,k}^N(r)|\leq  |\usigma_{i,k}(0)^N-\bar{\usigma}_{i,k}^N(0)|\\ 
& &\qquad+ \int_0^t\int_0^\infty\left|\psi(\usigma_{i,k}^N(s-))\mathbbm{1}_{(0,\bm{\lambda}^N_{i,k}(s-)]}(u) 
- \psi(\bar{\usigma}_{i,k}^N(s-))\mathbbm{1}_{(0,\bar{\bm{\lambda}}_{i,k}(s-)]}(u)\right|\mathcal{N}^{i,k}(\de u,\de s). 
\end{eqnarray*}

Taking the expectation on both sides, and recalling that the compensator of $\mathcal{N}^{i,k}$ is given by the Lebesgue measure, we obtain:

\begin{multline}\label{SupSigma}
E\Big[\sup_{r\in[0,t]}\left|\usigma_{i,k}^N(r)-\bar{\usigma}_{i,k}^N(r)\right| \Big]\leq \Lambda_{i,k}^N\\
+E\left[ \int_0^t\int_0^\infty\left|\psi(\usigma_{i,k}^N(s-))\mathbbm{1}_{(0,\bm{\lambda}^N_{i,k}(s-)]}(u)
- \psi(\bar{\usigma}_{i,k}^N(s-))\mathbbm{1}_{(0,\bar{\bm{\lambda}}_{i,k}(s-)]}(u)\right|\de u\de s\right]. 
\end{multline}

Consider the integral in the expectation above 
\begin{multline*}
\int_0^t\int_0^\infty\left|\psi(\usigma_{i,k}^N(s-))\mathbbm{1}_{(0,\bm{\lambda}^N_{i,k}(s-)]}(u)
- \psi(\bar{\usigma}_{i,k}^N(s-))\mathbbm{1}_{(0,\bar{\bm{\lambda}}_{i,k}(s-)]}(u)\right|\de u\de s \\
\leq \int_0^t\int_0^\infty 
\left|\psi(\usigma_{i,k}^N(s-))\big(\mathbbm{1}_{(0,\bm{\lambda}^N_{i,k}(s-)]}(u)
- \mathbbm{1}_{(0,\bar{\bm{\lambda}}_{i,k}(s-)]}(u)\big)\right|\de u\de s   \\ 
+\int_0^t\int_0^\infty 
\left|\psi(\usigma_{i,k}^N(s-))
- \psi(\bar{\usigma}_{i,k}^N(s-))\right|\mathbbm{1}_{(0,\bar{\bm{\lambda}}_{i,k}(s-)]}(u) \de u\de s.
\end{multline*}

Observe that: 
\begin{eqnarray*}
& & \left|\bm{\lambda}^N_{i,k}(s-)-\bar{\bm{\lambda}}_{i,k}(s-)\right|\leq 
2 \|\phi_k\|_{\infty}\Big\{\left| \usigma_{i,k}^N(s-)
- \bar{\usigma}_{i,k}^N(s-) \right|+ 
\sum_{h=1,2} \Big|\frac{N_h}{N}\bm{m}^N_h(s-)-r_h m_h(s-) \Big|\Big\}\nonumber\\ & &\hspace{4cm}+ \|\phi_k\|_{\infty}\sum_{h=1,2} \Big|\frac{N_h}{N}-r_h \Big| 
\end{eqnarray*}
and 
\begin{eqnarray*}
\bar{\bm{\lambda}}_{i,k}(s-)\left|\psi(\usigma_{i,k}^N(s-))
- \psi(\bar{\usigma}_{i,k}^N(s-))\right| 
\leq 2\|\phi_k \|_\infty |\usigma_{i,k}^N(s-)-\bar{\usigma}_{i,k}^N(s-) |. 
\end{eqnarray*}

Then, from \eqref{SupSigma} we obtain: 
\begin{multline}\label{SupSigma2}
E\Big[\sup_{r\in[0,t]}\left|\usigma_{i,k}^N(r)-\bar{\usigma}_{i,k}^N(r)\right| \Big]\leq \Lambda_{i,k}^N
+4 \|\phi_k \|_\infty \int_0^t E\big[\left|\usigma_{i,k}^N(s-)-\bar{\usigma}_{i,k}^N(s-)\right| \big]\de s\\
+ 2 \|\phi_k \|_\infty \int_0^t \sum_{h=1,2} E\Big[\Big|\frac{N_h}{N}\bm{m}^N_h(s-)-r_h m_h(s-) \Big|\Big]\de s 
+ t \|\phi_k \|_\infty  \sum_{h=1,2} \Big|\frac{N_h}{N}-r_h \Big|. 
\end{multline}

Moreover, if we set $\bar{\bm{m}}_k^N=\frac{1}{N_k}\sum_{j=1}^{N_k}\bar{\usigma}_{j,k}^N$ we have:

\begin{multline} \label{triangular inequality1}
E\Big[ \Big|\frac{N_{k}}{N}\bm{m}_{k}^N(s-)-r_{k}m_{k}(s-)\Big|\Big] \leq E\Big[ \Big|\frac{N_{k}}{N}\bm{m}_{k}^N(s-)-\frac{N_k}{N}\bar{\bm{m}}_k^N(s-)\Big|\Big] +E\Big[ \Big|\frac{N_{k}}{N}\bar{\bm{m}}_{k}^N(s-)-r_{k}m_{k}(s-)\Big|\Big]\\
\leq \frac{N_k}{N} \frac{1}{N_k}\sum_{j=1}^{N_k}E\Big[ \Big|\usigma_{j,k}^N(s-)-\bar{\usigma}_{j,k}^N(s-)\Big| \Big]+E\Big[ \Big|\frac{N_{k}}{N}\bar{\bm{m}}_{k}^N(s-)-r_{k}m_{k}(s-)\Big|\Big]\\
=  \frac{N_k}{N}E\Big[\Big| \usigma_{i,k}^N(s-)-\bar{\usigma}_{i,k}^N(s-)\Big| \Big]+E\Big[ \Big|\frac{N_{k}}{N}\bar{\bm{m}}_{k}^N(s-)-r_{k}m_{k}(s-)\Big|\Big]
\end{multline}
where the last equality holds by symmetry. 

Now, fix $i_1\in \{1,\ldots N_1\}$ and $i_2\in \{1,\ldots, N_2\}$. Using \eqref{SupSigma2} and \eqref{triangular inequality1} we obtain:  

\begin{multline*}
\sum_{h=1,2} 
E\Big[\sup_{r\in[0,t]}\left|\usigma_{i_h,h}^N(r)-\bar{\usigma}_{i_h,h}^N(r)\right| \Big]\leq
8C \int_0^t \sum_{h=1,2}E\Big[\sup_{r\in [0,s]}\left|\usigma_{i_h,h}^N(r)-\bar{\usigma}_{i_h,h}^N(r)\right| \Big]\de s\\
+ \sum_{h=1,2}\left\{\Lambda_{i_h,h}^N + 2 t C \Big|\frac{N_h}{N}-r_h \Big|   + 4 C \int_0^t E\Big[ \Big|\frac{N_{h}}{N}\bar{\bm{m}}_{h}^N(s-)-r_{k}m_{h}(s-)\Big|\Big]\de s\right\}
\end{multline*}
where $C=\|\phi_1\|_{\infty}\vee \|\phi_2\|_{\infty}$. 
By Gronwall inequality we have:

\begin{multline*}
\sum_{h=1,2} 
E\Big[\sup_{r\in[0,t]}\left|\usigma_{i_h,h}^N(r)-\bar{\usigma}_{i_h,h}^N(r)\right| \Big] \\
\leq e^{8Ct}\sum_{h=1,2}\left\{\Lambda_{i_h,h}^N + 2 t C \Big|\frac{N_h}{N}-r_h \Big|   + 4 C \int_0^t E\Big[\Big|\frac{N_{h}}{N}\bar{\bm{m}}_{h}^N(s-)-r_{h}m_{k}(s-)\Big|\Big]\de s\right\}.
\end{multline*}

By the hypothesis on $\Lambda_{i_h,h}^N$  
and the law of large numbers for $\{\bar{\bm{m}_h}^N; N>1\}$, choosing $i_h=i$ for $h=k$ and $t=T$ in the above inequality 
we obtain \eqref{convergence of mean value}.

\end{proof}

\section{Cooperative vs. competitive families:  fluctuations around the mean field limit}
From now on we focus on the case when $\phi_1^\prime \phi_2^\prime <0$. 
Our aim in next sections is to investigate how the microscopic dynamics fluctuates around its mean field approximation before it reaches its absorbing states.  
As observed in section \ref{section2}, the macroscopic system has an Hamiltonian $H$ that is conserved on the mean-field orbits  $\mathcal{C}_k, k\in (-\infty,0)$. Then, $H$ may be considered as a radial coordinate and we can change variables in such a way to represent the system through ''action-angle'' variables $(H, \Theta)$ (see \cite{arnold}).\smallskip\\
\noindent Consider the macroscopic equation \eqref{macro} for $(m_1, m_2)\in (0,1)^2$. Even though our results will be proved for general monotonic functions $\phi_1, \phi_2$, let us restrict for the moment to a simpler case for which we can write explicit formulas. Set  
\begin{equation}\label{Linearphi}
\phi_1(z)=az+b_1,\quad\quad \phi_2(z)=-az+b_2
\end{equation}
where $a>0$ and $b_1, b_2$ are such that the two functions are positive. In this case $\psi_1(z)= r_2a(2z-1)$ and $\psi_2(z)=-r_1a(2z-1)$ and the Hamiltonian is given by $H(m_1,m_2)=a\ln\big(m_1(1-m_1)m_2(1-m_2)\big)$.\\
We first change variables in order to shift the point $(\frac{1}{2}, \frac{1}{2})$ at the origin setting $x=m_1-\frac{1}{2},\ y=m_2-\frac{1}{2}$. For $(x, y)\in \left(-\frac{1}{2}, \frac{1}{2}\right)^2\setminus \{(0,0)\}$, the macroscopic equation (\ref{macro}) becomes:
\begin{equation}\label{macro2}
 \begin{array}{cc} 
\dot{x}=& 2ar_1r_2y(\frac{1}{4}-x^2)\\
\dot{y}=& -2ar_1r_2x(\frac{1}{4}-y^2)
\end{array}
\end{equation}
Taking the equivalent Hamiltonian $e^{a^{-1}H(m_1,m_2)}$
and, with an abuse of notation, denoting it again by $H$ we consider the change of variables given by: 
\begin{eqnarray}
H(x,y)&=& \left(\frac{1}{4}-x^2\right)\left(\frac{1}{4}-y^2\right);\label{H(x,y) in 0,0}\\
\Theta(x,y)&=&\left\{ \begin{array}{ccc}
\arctan\frac{y}{x}  &\mbox{if } x>0, y\geq 0,\\
\frac{\pi}{2}   &\mbox{if }   x=0, y>0,\\
\arctan\frac{y}{x}+\pi   &\mbox{if }  x<0, \\
\frac{3\pi}{2} &\mbox{if } x=0,\ y<0,\\
\arctan\frac{y}{x}+2\pi  &\mbox{if }\ x>0,\ y<0\\\label{thetaMacro}
\end{array}
\right.
\end{eqnarray}
where $H\in \left(0, \frac{1}{16}\right),\  \Theta \in \mathbb{R}/2\pi \mathbb{Z}$. The derivative of $\Theta$ is given by
\begin{equation*}
\dot{\Theta}=\frac{1}{1+(\frac{y}{x})^2}\left(-\frac{y}{x^2}\dot{x}+\frac{1}{x}\dot{y}\right) = -\frac{2ar_1r_2}{1+(\frac{y}{x})^2}\left[ \left(\frac{y}{x}\right)^2\left(\frac{1}{4}-x^2\right)+\left(\frac{1}{4}-y^2\right)\right] .
\end{equation*}
For $x=0$ we have $y=\pm \frac{1}{2}\sqrt{1-16H}$. For $x\neq 0$, replacing $y^2=x^2\tan^2 \Theta$ in (\ref{H(x,y) in 0,0}) and recalling that $x^2\in \left(0, \frac{1}{4}\right)$, we get 
\begin{eqnarray*}
x^2=& \frac{1}{2\tan^2\Theta}\Big( \frac{\tan^2\Theta+1}{4}-\sqrt{(\frac{\tan^2\Theta+1}{4})^2-4\tan^2\Theta(\frac{1}{16}-H)}\Big)
\end{eqnarray*}    
from which we obtain: 
\begin{eqnarray*}
x^2&=& \frac{1}{8\sin^2\Theta}\Big(1-\sqrt{1-\sin^2(2\Theta)(1-16H)}\Big);\\
y&=& x\tan \Theta.
\end{eqnarray*}
Equation (\ref{macro2}) in the new coordinates is thus given by:
\begin{equation*}
\left\{\begin{array}{ll}
\dot{H}=0\\
\dot{\Theta}=-\frac{ar_1r_2}{2}\sqrt{1-\sin^2(2\Theta)(1-16H)}
\end{array}\right.
\end{equation*}
For each fixed $H\in (0, \frac{1}{16})$, the second equation can be solved by separating variables to obtain $F(2\Theta | 1-16H)=-ar_1r_2 t+ k(\Theta_0)$ , where $F(\varphi | m)=\int_0^\varphi \frac{d\theta}{\sqrt{1-m\sin^2\theta}}$ is the Legendre's elliptic integral of the first kind with amplitude $\varphi$ and parameter $m$, and $k(\Theta_0)$ denotes a constant which depends on the initial condition.  
The solution is then given by $\Theta(t)=-\frac{1}{2}\mathrm{am}_{1-16H}(ar_1r_2 t+ k(\Theta_0))$, with $\mathrm{am}_m(u)$ denoting the inverse function to $F(\varphi | m)$, known as \emph{Jacobi amplitude function} (see \cite{Whittaker}).\\
Note that, if we pose $F(H, \Theta)=\frac{ar_1r_2}{2}\sqrt{1-\sin^2(2\Theta)(1-16H)}$, for each fixed $H=h$ with $h\in (0, 1/16)$, 
equation $\dot{\Theta}=-F(h, \Theta)$ 
  generates an ergodic dynamical system with invariant distribution  
$\mu^h(\de \Theta)=\frac{1}{T(h)F(h,\Theta)} \de \Theta$, where 
$T(h)=\int_0^{2\pi} \frac{1}{F(h,\Theta)} \de \Theta $  is the period of the motion. \smallskip\\
The above representation suggests that the microscopic dynamics 
may be described in terms of a slow motion
(of the microscopic variable corresponding to $H$) and a faster one (of the variable corresponding to $\Theta$).  In particular, 
assuming that the fast motion has an invariant distribution $\mu^x$ for each fixed value $x$ of the slow component $H$, we expect that on the ''larger''  time scale at which the slow motion of $H$ is observable, the fast variable $\Theta$ \emph{averages out}. This means that on such time scale, for $N$ large enough, the dynamics becomes essentially one dimensional, being described by $H$, and its dependence on $\Theta$  should appear as an integral with respect to the measure $\mu^H$. 

\subsection{A change of variables for the microscopic system}
In the light of what we have discussed above, we shall give a new representation of the microscopic system by introducing two variables $(H, \Theta)$. The resulting markovian dynamics has a  generator whose form shows that such variables evolve on different time scales. 
\begin{prop}\label{prop:GENH}
For $N>1$, let $(\bm{x}^N, \bm{y}^N)$ be the process defined by  
 $\bm{x}^N(t)=\bm{m}_1^N(t)-1/2,\  \bm{y}^N(t)=\bm{m}_2^N(t)-1/2$ and let  $\varphi:D\to (-\infty, 0)\times \mathbb{R}/2\pi\mathbb{Z}$, with $D=\left(-\frac{1}{2}, \frac{1}{2}\right)^2-\{(0,0)\}$,  be the change of variables defined by $\varphi (x, y)=(H(x,y), \Theta(x,y))$, where $\Theta$ is the function defined in \eqref{thetaMacro} and 
\begin{equation*}
H(x, y)=\int_0^x \frac{\psi_2(\frac{1}{2}+z)}{r_1 (\frac{1}{4}-z^2)}\de z +\int_0^y \frac{-\psi_1(\frac{1}{2}+z)}{r_2 (\frac{1}{4}-z^2)}\de z. 
\end{equation*}
We pose $\varphi^{-1}(h,\theta)=(x(h,\theta), y(h,\theta))$.\\
Consider the process $(\bm{H}^N, \bm{\Theta}^N)$ defined by 
$(\bm{H}^N(t), \bm{\Theta}^N(t))=\varphi(\bm{x}^N(t\wedge\tau^N), \bm{y}^N(t\wedge\tau^N))$, where 
$\tau^N=\inf \{ t\geq 0 : (\bm{x}^N(t), \bm{y}^N(t))\notin D\}$.\\ 
Define , for any $\epsilon, \epsilon^\prime  >0$,  
$\tau_{-\frac{1}{\epsilon}, -\epsilon^\prime}^N=\inf \{t\geq 0: \bm{H}^N(t)\notin \left(-\frac{1}{\epsilon}, -\epsilon^\prime\right)\}$. Then, the stopped process 
$\left\{ \left( \bm{H}^N(t\wedge \tau_{-\frac{1}{\epsilon}, -\epsilon^\prime}^N), \bm{\Theta}^N(t\wedge \tau_{-\frac{1}{\epsilon}, -\epsilon^\prime}^N)\right)\right\}_{t\geq 0}$ has a generator of the form: 
\begin{multline}\label{generator micro001}
\mathcal{K}^{\epsilon,\epsilon^\prime}_Nf(h,\theta)= \frac{1}{N}
\left\{a^H(h,\theta) f_h(h,\theta)+a^{HH}(h,\theta)f_{hh}(h,\theta)+
a^{H\Theta}(h,\theta)f_{h\theta}(h,\theta) \right.\\
\left.\qquad  +\left[ N (-F(h,\theta)+o(1)) + G(h,\theta) \right]f_\theta(h,\theta) +a^{\Theta\Theta}(h,\theta)f_{\theta\theta}(h,\theta)\right\}\mathbbm{1}_{\left(-\frac{1}{\epsilon}, -\epsilon^\prime\right)}(h)  + o(1)
\end{multline}
where $f$ is a $\mathcal{C}^3$ function on $\left[-\frac{1}{\epsilon}, -\epsilon^\prime\right]\times \mathbb{R}/2\pi\mathbb{Z}$, $\lim_{N\to \infty}o(1)=0$ and $a^H, a^{HH}, F, G,$ $ a^{H\Theta}, a^{\Theta\Theta}$ are regular functions (at least $\mathcal{C}^1$) obtained by the coefficients in formula \eqref{GN0} by taking $x=x(h,\theta), y=y(h,\theta)$.  

\end{prop}
\begin{proof}
We recall that, for $i,j=1,2$:
\begin{eqnarray*}
\textstyle\psi_i \left(\frac{1}{2}+z\right)&=&\textstyle\phi_i\left(r_j(\frac{1}{2}+z)\right)-\phi_i\left(r_j(\frac{1}{2}-z)\right).
\end{eqnarray*}
In what follows, for $i,j=1,2$, we use the notations: 
\begin{eqnarray*}
\textstyle\psi_i^N\left(\frac{1}{2}+z\right)&=&\textstyle\phi_i\left(\frac{N_j}{N}(\frac{1}{2}+z)\right)-\phi_i\left(\frac{N_j}{N}(\frac{1}{2}-z)\right),\\
\textstyle\psi_i^{N+}\left(\frac{1}{2}+z\right)&=&\textstyle\phi_i\left(\frac{N_j}{N}(\frac{1}{2}+z)\right)+\phi_i\left(\frac{N_j}{N}(\frac{1}{2}-z)\right),\\
\textstyle\psi_i^+ \left(\frac{1}{2}+z\right)&=&\textstyle\phi_i\left(r_j(\frac{1}{2}+z)\right)+\phi_i\left(r_j(\frac{1}{2}-z)\right).
\end{eqnarray*}
  
For $N>1$, let $(\bm{x}^N, \bm{y}^N)$ be the process defined by  
 $\bm{x}^N(t)=\bm{m}_1^N(t)-1/2,\  \bm{y}^N(t)=\bm{m}_2^N(t)-1/2$ and $\tau^N=\inf \{t\geq 0 : (\bm{x}^N(t), \bm{y}^N(t))\notin D\}$. The generator of the process $(\bm{x}^N(\cdot\wedge\tau^N), \bm{y}^N(\cdot\wedge\tau^N))$  for a function $g:\bar{E}^N \to\mathbb{R}$, with 
$\bar{E}^N=\{(x_1,x_2)\in [-\frac{1}{2},-\frac{1}{2}]^2: x_k=\frac{j}{N_k}-\frac{1}{2},  0\leq j\leq N_k, k=1,2 \}$,  is given by:

\begin{eqnarray*}
\mathcal{G}_N g (x,y)&=& \left\{N_1(\frac{1}{4}-x^2)\frac{N_1}{N} \phi_1\left(\frac{N_{2}}{N}(\frac{1}{2}+y)\right)\left[ g\left(x+\frac{1}{N_1}, y\right)-g(x,y)\right]\right.\\
&+& N_1(\frac{1}{4}-x^2)\frac{N_1}{N}\phi_1\left(\frac{N_{2}}{N}(\frac{1}{2}-y)\right)\left[ g\left(x-\frac{1}{N_1}, y\right)-g(x, y)\right] \\
&+& N_2(\frac{1}{4}-y^2)\frac{N_2}{N}\phi_2\left(\frac{N_{1}}{N}(\frac{1}{2}+x)\right)\left[g \left(x, y+\frac{1}{N_2}\right)-g(x, y)\right]\\
&+&\left. N_2(\frac{1}{4}-y^2)\frac{N_2}{N}\phi_2\left(\frac{N_{1}}{N}(\frac{1}{2}-x^N)\right)\left[ g\left(x, y-\frac{1}{N_2}\right)-g(x, y)\right]\right\}\mathbbm{1}_D(x,y).
\end{eqnarray*}
Let $(\bm{H}^N, \bm{\Theta}^N) $ be the Markov process defined by $(\bm{H}^N(t), \bm{\Theta}^N(t))=\varphi(\bm{x}^N(t\wedge\tau^N), \bm{y}^N(t\wedge \tau^N))$. Notice that $(\bm{H}^N, \bm{\Theta}^N)$ 
has as absorbing states all the points of  the form $(-\infty, \theta)$ with $\theta\in \mathbb{R}/2\pi\mathbb{Z}$ and 
the state corresponding to $x=y=0$, which can be identified with the point $(0,0)$.\\    
Define, for $\epsilon, \epsilon^\prime >0$, \  $D_{\epsilon, \epsilon^\prime}=\varphi^{-1}\left[(-\frac{1}{\epsilon}, -\epsilon^\prime) \times \mathbb{R}/2\pi\mathbb{Z}\right]$
and $\tau_{-\frac{1}{\epsilon}, -\epsilon^\prime}^N=\inf \{t\geq 0: \bm{H}^N\notin (-\frac{1}{\epsilon},-\epsilon^\prime)\}$.\\  
Now, consider the stopped process $\big(\bm{H}^N(\cdot\wedge \tau^N_{-\frac{1}{\epsilon}, -\epsilon^\prime}), \bm{\Theta}^N(\cdot\wedge\tau^N_{-\frac{1}{\epsilon}, -\epsilon^\prime})\big)$  and let $\mathcal{G}_N^{\epsilon, \epsilon^\prime}$ be the generator obtained from $\mathcal{G}_N$ by replacing $\mathbbm{1}_D$ with $\mathbbm{1}_{D_{\epsilon, \epsilon^\prime}}$. Let us apply $\mathcal{G}_N^{\epsilon,\epsilon^\prime}$ to $g=f\circ \varphi$ where $f: [-\frac{1}{\epsilon^\prime}, -\epsilon]\times \mathbb{R}/2\pi\mathbb{Z}\to \mathbb{R}$ is a smooth function. 
Then, with the usual notations $(\cdot)_x:= \partial/\partial x$ for the partial derivatives we can write:  
\begin{eqnarray}\label{GNepsilon}
\mathcal{G}_N^{\epsilon, \epsilon^\prime} (f\circ \varphi)(x,y)&=& \left\{\frac{N_1}{N} \left(\frac{1}{4}-x^2\right)\psi^N_1(\frac{1}{2}+y) (f\circ \varphi)_x 
+\frac{N_1}{N} (\frac{1}{4}-x^2)\psi^{N+}_1(\frac{1}{2}+y)  \frac{1}{N_1}(f\circ \varphi)_{xx}\right.\nonumber\\
&+&\left.\frac{N_2}{N} (\frac{1}{4}-y^2)\psi^N_2(\frac{1}{2}+x) (f\circ \varphi)_y 
+\frac{N_2}{N} (\frac{1}{4}-y^2) \psi^{N+}_2(\frac{1}{2}+x)  \frac{1}{N_2}(f\circ \varphi)_{yy}\right\}\mathbbm{1}_{D_{\epsilon, \epsilon^\prime}}(x,y)\nonumber\\
&+& R_N^{\epsilon, \epsilon^\prime}(x,y) 
\end{eqnarray}
where $\lim_{N\to \infty}N R^{\epsilon, \epsilon^\prime}_N(x,y)=0$. The above derivatives are given by:

\begin{align*}
(f\circ \varphi)_x & = \frac{\psi_2(\frac{1}{2}+x)}{r_1(\frac{1}{4}-x^2)}f_h-\frac{y}{x^2+y^2}f_\theta;\\
(f\circ \varphi)_{xx} & =\left[ \frac{\psi_2(\frac{1}{2}+x)}{r_1(\frac{1}{4}-x^2)}\right]^2 f_{hh}
+\left\{\frac{(\psi_2)^\prime(\frac{1}{2}+x)}{r_1(\frac{1}{4}-x^2)}+\frac{\psi_2(\frac{1}{2}+x)2x}{r_1(\frac{1}{4}-x^2)^2}\right\}  f_h -\frac{\psi_2(\frac{1}{2}+x)2y}{r_1(\frac{1}{4}-x^2)(x^2+y^2)}
f_{h\theta} 
\\ &+ \frac{y^2}{(x^2+y^2)^2}f_{\theta\theta} +\frac{2xy}{(x^2+y^2)^2}f_\theta. \\
(f\circ \varphi)_y & =\frac{-\psi_1(\frac{1}{2}+y)}{r_2(\frac{1}{4}-y^2)}f_h+\frac{x}{x^2+y^2}f_\theta; \\
(f\circ \varphi)_{yy} & =\left[ \frac{\psi_1(\frac{1}{2}+y)}{r_2(\frac{1}{4}-y^2)}\right]^2 f_{hh}
+\left\{\frac{-(\psi_1)^\prime(\frac{1}{2}+y)}{r_2(\frac{1}{4}-y^2)}-\frac{\psi_1(\frac{1}{2}+y)2y}{r_2(\frac{1}{4}-y^2)^2}\right\} f_h 
-\frac{\psi_1(\frac{1}{2}+y)2x}{r_2(\frac{1}{4}-y^2)(x^2+y^2)}f_{h\theta}\\ 
&+ \frac{x^2}{(x^2+y^2)^2}f_{\theta\theta} -\frac{2xy}{(x^2+y^2)^2}f_\theta.
\end{align*}

Now observe that, by the regularity of the functions $\phi_1,\phi_2$ and $f\circ\varphi $ on the compact set $\bar{D}_{\epsilon, \epsilon^\prime}$ (where $\bar{A}$ denotes the closure of a set $A$), we have $\psi_i^N(z)=\psi_i(z)+o(1)$, $\psi^{N+}_i(z)=
\psi^+_i(z)+o(1),\ \ i=1,2$, $\sup_{(x,y)\in D_{\epsilon, \epsilon^\prime}}R_N^{\epsilon, \epsilon^\prime}(x,y)\leq o(\frac{1}{N})$  and \eqref{GNepsilon} can be written as follows: 
\begin{eqnarray}\label{GN0}
& &\mathcal{G}_N^{\epsilon, \epsilon^\prime} (f\circ \varphi)(x,y)=
\left\{\frac{1}{N} \left(
\frac{\psi_1^{+}(\frac{1}{2}+y)}{r_1}\Big[(\psi_2)^\prime(\frac{1}{2}+x)+\frac{\psi_2(\frac{1}{2}+x)2x}{\frac{1}{4}-x^2}\Big]\right.\right.\nonumber\\
& &\qquad\qquad\qquad\qquad\qquad\qquad \left. -\frac{\psi_2^{+}(\frac{1}{2}+x)}{r_2}\Big[(\psi_1)^\prime(\frac{1}{2}+y)+\frac{\psi_1(\frac{1}{2}+y)2y}{\frac{1}{4}-y^2}\Big]\right) f_h \nonumber\\ 
& &\qquad +\frac{1}{N}\left(\frac{\left[\psi_2(\frac{1}{2}+x)\right]^2\psi^{+}_1(\frac{1}{2}+y)}{r_1^2 (\frac{1}{4}-x^2)}+ \frac{\left[\psi_1(\frac{1}{2}+y)\right]^2 \psi^{+}_2(\frac{1}{2}+x)}{r_2^2(\frac{1}{4}-y^2)} \right)f_{hh}\nonumber\\
& &\qquad -\frac{1}{N}\left(\frac{\psi_2(\frac{1}{2}+x)\psi^{+}_1(\frac{1}{2}+y)2y}{r_1(x^2+y^2)}+\frac{\psi_1(\frac{1}{2}+y)\psi^{+}_2(\frac{1}{2}+x)2x}{r_2(x^2+y^2)} \right)f_{h\theta}\\
& &\qquad +\left(\frac{1}{N}\frac{2xy[(\frac{1}{4}-x^2)\psi_1^{+}(\frac{1}{2}+y)- (\frac{1}{4}-y^2)\psi_2^{+}(\frac{1}{2}+x)]}{(x^2+y^2)^2}\right.\nonumber\\
& &\qquad\qquad\qquad \left. - \frac{r_1\psi_1(\frac{1}{2}+y)y(\frac{1}{4}-x^2)- r_2\psi_2(\frac{1}{2}+x)x(\frac{1}{4}-y^2)}{x^2+y^2}+o(1)\right)f_\theta\nonumber\\ 	
& &\left.\qquad+\frac{1}{N}\frac{(\frac{1}{4}-x^2)\psi_1^{+}(\frac{1}{2}+y)y^2+ (\frac{1}{4}-y^2)\psi_2^{+}(\frac{1}{2}+x)x^2}{(x^2+y^2)^2}f_{\theta\theta}\right\}\mathbbm{1}_{D_{\epsilon, \epsilon^\prime}}(x,y)+o\left(\frac{1}{N}\right).\nonumber
\end{eqnarray}
We  rewrite \eqref{GN0} as: 
\begin{eqnarray*}
\mathcal{G}_N^{\epsilon, \epsilon^\prime} (f\circ \varphi)=
\frac{1}{N}\left\{a^{(h)}f_h+a^{(hh)}f_{hh}+ a^{(\theta)} f_\theta 
+   a^{(h\theta)} f_{h\theta}+ a^{(\theta\theta)}f_{\theta\theta}\right\}\mathbbm{1}_{D_{\epsilon, \epsilon^\prime}}+ o(\frac{1}{N}).
\end{eqnarray*} 
Using the inverse change of variables $\varphi^{-1}(h,\theta)=\left(x(h,\theta), y(h,\theta)\right)$, the above expression can be written in terms of the variables $(h, \theta)$. We pose $a^{H}(h, \theta)=a^{(h)}(\varphi^{-1}(h, \theta))$ and define analogously $a^{HH}, a^{\Theta}, a^{H\Theta}, a^{\Theta\Theta}$.  
Then, using \eqref{GN0},  
we can write the asymptotic of the generator $\mathcal{K}_N^{\epsilon, \epsilon^\prime}$ of  
$\big(\bm{H}^N(\cdot\wedge\tau^N_{-\frac{1}{\epsilon},-\epsilon^\prime}),\bm{\Theta}^N(\cdot\wedge\tau^N_{-\frac{1}{\epsilon},-\epsilon^\prime})\big)$ on $f$ for large $N$: 

\begin{eqnarray}\label{GEN01}
\mathcal{K}_N^{\epsilon, \epsilon^\prime} f (h,\theta)=\frac{1}{N}\left\{a^{H}(h,\theta) f_h(h,\theta)+a^{HH}(h,\theta)f_{hh}(h,\theta)+ a^{\Theta}(h,\theta) f_\theta(h,\theta) \right.\nonumber\\
\qquad\qquad\qquad\qquad  \left. + a^{H\Theta}(h,\theta) f_{h\theta}(h,\theta)+ a^{\Theta\Theta}(h,\theta)f_{\theta\theta}(h,\theta)\right\}\mathbbm{1}_{\left(-\frac{1}{\epsilon}, -\epsilon^\prime\right)}(h) + o(\frac{1}{N})
\end{eqnarray}
where, to emphasize the presence of the term of order $N$ in $a^\Theta$, we can write $a^\Theta(h,\theta)= N (-F(h,\theta)+o(1))+G(h,\theta)$.
\end{proof}

Note that by \eqref{GEN01} and the form of $a^{\Theta}$ we obtain the macroscopic dynamics \eqref{macro} in terms of the new variables: 
\begin{equation*}
\left\{\begin{array}{ll}
\dot{H}=0;\\
\dot{\Theta}=- F(H,\Theta).
\end{array}
\right.
\end{equation*}
In next subsection we shall prove that for each fixed $h\in (-\infty, 0)$, the function $F(h, \cdot)$ is $\mathcal{C}^1$ and bounded from below by a positive constant $c(h)$ (see the first part of the proof of Proposition \ref{existenceuniqueness}), so that the dynamics $\dot{\Theta}=-F(h, \Theta)$ has a unique invariant distribution given by $\mu^h(\de \theta)=\frac{1}{\mathcal{T}(h)}\frac{1}{F(h,\theta)}\de \theta$ with $\mathcal{T}(h)$ being the normalizing constant. 

\subsection{Main result}\label{s:mainresult}

As can be seen by the coefficients in  \eqref{generator micro001}, the term of order $1$ which appears in $a^\Theta$  indicates that the variable  $\bm{H}^N$ 
jumps at a time scale larger than the one of  $\bm{\Theta}^N$.\\
The goal of this subsection is to describe the macroscopic behaviour of 
the process $\{\bm{H}^N(Nt)\}_{t\in [0,T]}$ as $N\to \infty$.\\ 
Let us consider the generator \eqref{generator micro001} and change the time scale by multiplying it by $N$. 
We obtain the following expression:   
\begin{eqnarray}\label{generator micro01}
& &\tilde{\mathcal{K}}^{\epsilon,\epsilon^\prime}_Nf= \left\{a^H f_h+a^{HH}f_{hh}+ 
 \left[ N (-F+o(1)) + G \right]f_\theta 
\right.\nonumber\\& & \left.\qquad\qquad\qquad\qquad\qquad\qquad+ a^{H\Theta}f_{h\theta}  +a^{\Theta\Theta}f_{\theta\theta}\right\}\mathbbm{1}_{\left(-\frac{1}{\epsilon}, -\epsilon^\prime\right)\times \mathbb{R}/2\pi\mathbb{Z}}  + o(1).
\end{eqnarray}
Next Theorem shows that, as $N\to \infty$, the 
process $\{\bm{H}^N(Nt)\}_{t\in [0,T]}$ behaves like the solution of a stochastic differential equation;  the coefficients of such equation are averages with respect to the invariant distribution for the macroscopic dynamics of the variable $\Theta$.\\
For $h\in (-\infty, 0)$, consider 
\begin{eqnarray*}
& &F(h,\theta) = \frac{r_1\psi_1\left(\frac{1}{2}+y(h,\theta)\right)y(h,\theta)\left(\frac{1}{4}-x(h,\theta)^2\right)- r_2\psi_2\left(\frac{1}{2}+x(h,\theta)\right)x(h,\theta)\left(\frac{1}{4}-y(h,\theta)^2\right)}{x(h, \theta)^2+y(h,\theta)^2},\\
& &\mathcal{T}(h)=\int_0^{2\pi} \frac{1}{F(h,\theta)}\de \theta
\end{eqnarray*}
 and define
 \begin{eqnarray}\label{coefficientsSDEa^H}
 & &\bar{a}^H(h):=\int_0^{2\pi}a^{H}(h,\theta)\mu^h(\de\theta)=\int_0^{2\pi}\frac{a^{H}(h, \theta)}{\mathcal{T}(h)F(h,\theta)}\de\theta,\\\label{coefficientsSDEa^HH}
& &\bar{a}^{HH}(h):=\int_0^{2\pi}a^{HH}(h, \theta)\mu^h(\de\theta)
=\int_0^{2\pi}\frac{a^{HH}(h, \theta)}{\mathcal{T}(h)F(h,\theta)}\de\theta.
\end{eqnarray}

\begin{theorem}\label{Main}
Let $\bar{a}^H, \bar{a}^{HH}$ be defined as in \eqref{coefficientsSDEa^H} and \eqref{coefficientsSDEa^HH} .
Fix $T>0$ and, for any $\epsilon>0$, let $\uptau_{-\frac{1}{\epsilon},0}^N:=\inf \{t\in [0, T]: \bm{H}^N(Nt)\notin (-\frac{1}{\epsilon}, 0)\}$. 
Then, the sequence of stopped processes $\Big\{ \big\{\bm{H}^N(Nt\wedge\uptau_{-\frac{1}{\epsilon}, 0}^N)\big\}_{t\in [0, T]}\ ; N>1\Big\}$  converges weakly, as $N\to\infty$, to the stopped process $\big\{\bm{H}(t\wedge\uptau_{-\frac{1}{\epsilon},0})\big\}_{t\in [0,T]}$, where $\bm{H}$ is the solution in $(-\frac{1}{\epsilon}, 0)$ of the SDE
\begin{equation}\label{equation SDE}
\de\bm{H}(t)=\bar{a}^H\big(\bm{H}(t)\big)\de t + \sqrt{\bar{a}^{HH}\big(\bm{H}(t)\big)}\de B(t)
\end{equation}
 with $\{B(t)\}_{t\geq 0}$ being a Brownian motion and $\uptau_{-\frac{1}{\epsilon},0}:= \inf \{t\in [0, T]: \bm{H}(t)\notin (-\frac{1}{\epsilon}, 0)\}$. 
\end{theorem}
The Theorem will be proved in subsection \ref{proofMain}, using the results of subsection \ref{averagingpr} and Proposition \ref{existenceuniqueness} of next paragraph.  
 
\paragraph{The limit process.}
In order to show that equation \eqref{equation SDE} is well posed and to state its poperties we shall use some known  results concerning existence and uniqueness of solutions of stochastic differential equations in an interval of the real line (see, e.g., \cite{Karatzas}, section 5.5, p. 329).\\
We recall the definition and a fundamental 
result.
\begin{Def}
Let $I=(l, r)$ be an interval of the real line. A \emph{weak solution in} $I$ of the equation \begin{equation}\label{SDE}
\de X(t)=b(X(t))\de t+\sigma(X(t))\de B(t)
\end{equation} is a pair 
$\bm{\Omega}=(\Omega, \mathcal{F}, \{\mathcal{F}_t\}, P)$,  
$(X, B)$, where  $\bm{\Omega}$ is a filtered probability space satisfying the usual conditions, $X$ is a continuous adapted process taking values in $[l,r]$ with $X(0)\in I$ a.s. and $B:=\{B(t), \mathcal{F}_t\}_{t\geq 0} $ is a standard Brownian motion, such that, for all $\bar{l}>l, \bar{r}<r$, letting $\tau_{\bar{l}, \bar{r}}:=\inf\{t\geq 0: X(t)\notin (\bar{l}, \bar{r})\}$ we have: 
\begin{itemize}
\item[ ] $P\left\{\int_0^{t\wedge \tau_{\bar{l}, \bar{r}}}\left[|b(X(s))|+\sigma^2(X(s))\right] \de s <\infty \right\}=1$ for all $t\geq 0$; 
\item[ ] $P\left\{X(t\wedge \tau_{\bar{l}, \bar{r}})=X(0)+\int_0^{t}b(X(s))\mathbbm{1}_{\{s\leq \tau_{\bar{l}, \bar{r}}\}}\de s+
\int_0^{t}\sigma(X(s))\mathbbm{1}_{\{s\leq \tau_{\bar{l}, \bar{r}}\}}\de B(s)\ \forall t\geq 0\right\}=1$.
\end{itemize}  
\end{Def}
We denote by $\tau_{I}$ the \emph{exit time from} $I$, i.e.,  
$$\tau_{I}=\lim_{n\to\infty} \tau_{l_n, r_n}$$
where $\{l_n\}$ and $\{r_n\}$ are strictly monotonic sequences with 
$l<l_n<r_n<r$ for all $n$ and $\lim_{n\to \infty}l_n=l,\  \lim_{n\to \infty}r_n=r$. 
\begin{theorem}\label{THMKS}(Thm 5.1 and subsection C of \cite{Karatzas} )\\
Suppose that the coefficients of \eqref{SDE} satisfy: 
\begin{eqnarray}
& &\sigma^2(x)>0,\quad  \forall x\in I;\label{CONDSDE1}\\
& &\forall x\in I\  \exists \epsilon>0 \mbox{ such that } \int_{x-\epsilon}^{x+\epsilon} \frac{1+|b(y)|}{\sigma^2(y)}\de y <\infty.\label{CONDSDE2}
\end{eqnarray} 
Then, for every initial distribution $\mu$ with $\mu(I)=1$, the equation \eqref{SDE} has a weak solution in $I$ and this solution is unique in the sense of probability law.
\end{theorem}

In next proposition we show that \eqref{equation SDE} has a  weak solution in $(-\infty, 0)$ and, for any $\epsilon>0$,  the solution in the interval $(-\tfrac{1}{\epsilon}, 0)$ exits a.s. from it and does so by the left side. 
\begin{prop}\label{existenceuniqueness}
For every initial distribution $\mu$ with $\mu\{(-\infty, 0)\}=1$ the equation \eqref{equation SDE} has a weak solution in the interval $I=(-\infty, 0)$ and this solution is unique in the sense of probability law. 
Moreover, if we let $\uptau_I=\inf\{t\geq 0 : \bm{H}(t)\notin I\}$ and, for all $\epsilon>0$, $\uptau_{-\frac{1}{\epsilon}, 0}=\inf\{t\geq 0 : \bm{H}(t)\notin (-\tfrac{1}{\epsilon}, 0)\}$ we have 
\begin{equation}\label{limitiI}
P\left(\lim_{t\to \uptau_I} \bm{H}(t\wedge \uptau_{I})=-\infty\right)=P\left(\sup_{0\leq t\leq \uptau_{I}}\bm{H}(t\wedge\uptau_{I}) <0\right)=1
\end{equation}
and 
\begin{equation}\label{limitiIepsilon}
P(\uptau_{-\frac{1}{\epsilon},0}<\infty)=1.
\end{equation}
\end{prop}
\begin{remark}
In next paragraph we illustrate the case when $\phi_1$ and $\phi_2$ are two linear functions. In this case we can obtain explicit expressions for the coefficients $\bar{a}^{H}$ and $\bar{a}^{HH}$ and the random time in \eqref{limitiIepsilon} will be replaced by $\uptau_I$. 
In the general case we have to restrict to the interval $I_\epsilon$.  
Indeed, since we cannot have explicit expressions of 
$a^{H}(h,\theta)$ and $a^{HH}(h,\theta)$ in terms of elementary functions, in order to obtain information about $\bm{H}$ near the endpoints of  $(-\infty, 0)$, we need estimates on such coefficients which are possible only when $h$ is close to 0.\\
However, we are interested in the behaviour of $\bm{H}$ \emph{before} 
it eventually reaches $-\infty$, since this should describe the behaviour of the microscopic variable $\bm{H}^N(Nt)$ for large $N$ before it reaches its absorbing state $-\infty$. Therefore, for our purposes it will be enough 
to study the process in the interval $\left(-\frac{1}{\epsilon}, 0\right)$ for $\epsilon$ arbitrarily small.  
\end{remark}
\begin{proof}
We recall that $a^H$ is defined by $a^H(h,\theta)=a^{(h)}(x(h,\theta), y(h,\theta))$, where $a^{(h)}(x,y)$ is the coefficient of $f_h$ in \eqref{GN0} and the analogous relation holds for $a^{HH}$. 
Let us fix two small positive numbers $\epsilon, \epsilon^\prime$ and suppose $h\in [-\frac{1}{\epsilon^\prime},-\epsilon ]$. 
Consider the term of order 1 in \eqref{GN0}, i.e., 
$$\tilde{F}(x,y) =- \frac{r_1\psi_1(\frac{1}{2}+y)y(\frac{1}{4}-x^2)- r_2\psi_2(\frac{1}{2}+x)x(\frac{1}{4}-y^2)}{x^2+y^2}$$ which has been written in \eqref{generator micro01} as $-F(h, \theta)=\tilde{F}(x(h,\theta), y(h,\theta))$. 
Note that, for $x,y\neq 0$ we have $\psi_1(\frac{1}{2}+y)y>0$ and 
$\psi_2(\frac{1}{2}+x)x>0$. Moreover, for  $x,y \in \bar{D}_{\epsilon, \epsilon^\prime}$ there exist $\delta=\delta(\epsilon)>0$ and $\delta^\prime=\delta^\prime(\epsilon^\prime)>0$ such that 
$\delta^2<(|x|\vee |y|)^2<\frac{1}{4}-\delta^\prime$, from which it follows that:
$$-\tilde{F}(x,y)\geq c(\epsilon, \epsilon^\prime)$$ for a constant $c(\epsilon, \epsilon^\prime)>0$. 
Then, the function $F$ is $\mathcal{C}^1$ (hence bounded) on $[-\frac{1}{\epsilon^\prime},-\epsilon]\times \mathbb{R}/2\pi \mathbb{Z}$ and the same holds for the functions $\frac{1}{F}$ and $\frac{1}{\mathcal{T}}$ where $\mathcal{T}:[-\frac{1}{\epsilon^\prime},-\epsilon]\to \mathbb{R}^+$ is given by $\mathcal{T}(h)=\int_{0}^{2\pi}\frac{1}{F(h, \theta)}\de \theta$. \\
Analogously, the functions $a^H(h,\theta)$ and $a^{HH}(h,\theta)$ are both of class $\mathcal{C}^1$ on the same interval, and so the functions $\bar{a}^H$ and $\bar{a}^{HH}$ are Lipschitz continuous for $h\in [-\frac{1}{\epsilon^\prime},-\epsilon]$ for all $\epsilon, \epsilon^\prime$. Moreover, the function $a^{(hh)}$ is strictly positive for all $(x, y)\in D$, hence the same holds for  $\bar{a}^{HH}$ for all $h\in \left(-\infty, 0\right)$. Then, conditions \eqref{CONDSDE1} and \eqref{CONDSDE2} of Theorem \ref{THMKS} are satisfied.\\

Now, for a fixed number $c\in I$, let us consider the \emph{scale function}
 \begin{equation}\label{scalefunction}
 p(z)=\int_c^z \exp\left\{-2\int_c^u\frac{\bar{a}^H(h)}{\bar{a}^{HH}(h)}\de h\right\}\de u,\ \ \ x\in I.
 \end{equation}
 Such function does not depend on the choice of $c$ and, 
according to  Proposition 5.22 of \cite{Karatzas}, a sufficient condition for \eqref{limitiI} is $\lim_{z\to -\infty}p(z)>-\infty$ and $\lim_{z\to 0^-}p(z)=\infty$. 

For the second limit, let us observe that the functions $a^{(h)}, a^{(hh)}$ in \eqref{GN0} are of class $\mathcal{C}^1$ and $\mathcal{C}^2$ respectively in a neighbourhood of $(0,0)$. Note also that $(x,y)\to 0$ if and only if $h\to 0$. Then, by Taylor expansion, for $(x,y)$ close to $(0,0)$ we have 
\begin{eqnarray*}
& &\textstyle  a^{(hh)}(x, y)=32 \{\phi_1(\frac{r_1}{2})[\phi_2^\prime(\frac{r_2}{2})]^2 x^2+ \phi_2(\frac{r_2}{2})[\phi_1^\prime(\frac{r_1}{2})]^2 y^2\}+o((x+y)^2);\\ 
& &\textstyle a^{(h)}(x,y)= -4 \{\phi_1(\frac{r_2}{2})|\phi_2^\prime(\frac{r_1}{2})|+\phi_2(\frac{r_1}{2})\phi_1^\prime(\frac{r_2}{2})\}+o((x+y)) 
\end{eqnarray*}
and 
\begin{eqnarray*}
\textstyle H(x, y)=-4\{|\phi_2^\prime(\frac{r_1}{2})|x^2+\phi_1^\prime(\frac{r_2}{2})y^2\}+o((x+y)^2).
\end{eqnarray*}
It follows that
$
a^{(hh)}(x,y)\leq -C_1 H(x,y) + o((x+y)^2)$ and 
$a^{(h)}(x,y)= -C_2 + o((x+y))$
with  
$C_1=8 [\phi_1(\frac{r_2}{2})|\phi_2^\prime(\frac{r_1}{2})|]\vee [\phi_2(\frac{r_1}{2})\phi_1^\prime(\frac{r_2}{2})]$ and 
$C_2=8 [\phi_1(\frac{r_2}{2})|\phi_2^\prime(\frac{r_1}{2})|+ \phi_2(\frac{r_1}{2})\phi_1^\prime(\frac{r_2}{2})]$ and so $C_2/C_1>1$.
Therefore, for $h$ close to 0, we have the following estimates, that hold uniformly in $\theta$: 
\begin{eqnarray*}
& &\textstyle a^{HH}(h,\theta)\leq -C_1 h+o(h);\\
& &\textstyle a^H(h, \theta)\leq -C_2+o(\sqrt{h}). 
\end{eqnarray*}
Integrating with respect to $\mu^h$ we obtain the same inequalities 
for $\bar{a}^{HH}(h)$ and $\bar{a}^H(h)$, and so, for $c<x<0$ and $c$ sufficiently close to 0 we have    
$p(z)\geq \int_c^z \exp\left\{\int_c^u\frac{C_2/C_1+o(\sqrt{h})}{-h+o(h)}\de h\right\}\de u\geq \int_c^z\exp\left\{\int_c^u \frac{C_2/C_1}{-h+o(h)}\de h \right\}du$ which implies $\lim_{z\to 0^-}p(z)=\infty$.\\ 
Now, for $h\leq c$ we have 
$\frac{1}{4}-(|x(h,\theta)|\vee|y(h,\theta)|)^2\leq \rho(c)$, with $\lim_{c\to -\infty}\rho(c)=0$,  and the following inequalities, that hold uniformly with respect to $\theta$:
\begin{eqnarray*}
& &\textstyle K_1 l\left(x(h, \theta),y(h,\theta)\right)\leq  a^{(HH)}(h, \theta) \leq K_2 l\left(x(h, \theta),y(h,\theta)\right)
;\\ 
& &\textstyle 2l\left(x(h, \theta),y(h,\theta)\right)\leq -2 a^{(H)}(h,\theta)\leq 2 (l\left(x(h,\theta),y(h,\theta)\right)+ K)
\end{eqnarray*}
where \begin{equation*}\textstyle l(x,y)=-\frac{\psi_1^+(\frac{1}{2}+y)\psi_2(\frac{1}{2}+x)2x}{r_1(\frac{1}{4}-x^2)}+
\frac{\psi_2^+(\frac{1}{2}+x)\psi_1(\frac{1}{2}+y)2y}{r_2(\frac{1}{4}-y^2)}
\end{equation*}
and  $K_1, K_2, K>0$ are constant; more precisely, $K_1=\frac{\inf |\phi_2^\prime|}{r_1}\wedge \frac{\inf \phi_1^\prime}{r_2},\ K_2=
\frac{\sup |\phi_2^\prime|}{r_1}\vee \frac{\sup \phi_1^\prime}{r_2},\ 
K=4\left(\frac{\sup\phi_1\sup |\phi_2^\prime|}{r_1}\vee \frac{\sup \phi_2\sup\phi_1^\prime}{r_2}\right)$. 
Then, integrating $a^H(h,\theta)$ and $a^{HH}(h, \theta)$ with respect to the measure $\mu^{h}$ and posing $\bar{l}(h)=\int_0^{2\pi}l(x(h,\theta), y(h,\theta))\mu^h(\de\theta)$,  we obtain: 
 \begin{eqnarray*}
& &\textstyle K_1 \bar{l}(h)\leq  \bar{a}^{(HH)}(h) \leq K_2 \bar{l}(h)   
;\\ 
& &\textstyle 2\bar{l}(h)\leq -2 \bar{a}^{(H)}(h)\leq 2 (\bar{l}(h)+ K)
\end{eqnarray*}
with $\lim_{h\to -\infty}\bar{l}(h)=\infty$. Therefore   
$p(z)\geq \int_c^z \exp\{\int_c^u \frac{2}{K_2}\de h\}\de u=\frac{K_2}{2}\left(-1+e^{\frac{2}{K_2}(z-c)}\right)$ form which it follows 
$\lim_{z\to-\infty}p(z)>-\infty$.\\
Finally, for $c, z\in \left(-\frac{1}{\epsilon}, 0 \right)$ consider the function 
\begin{equation}\label{vscalefunction}
v(z)=\int_c^{z}\exp\left\{-2\int_c^y \frac{\bar{a}^H(h)}{\bar{a}^{HH}(h)}\de h\right\}\left(\int_c^y \frac{2}{\exp\left\{-2\int_c^w \frac{\bar{a}^H(h)}{\bar{a}^{HH}(h)}\de h\right\}\bar{a}^{HH}(w)}\de w\right) \de y .
\end{equation}
By Proposition 5.32 of \cite{Karatzas}, $P(\uptau_{-\frac{1}{\epsilon}, 0 }<\infty)=1$ if  $\lim_{z\to 0^-}p(z)=\infty$ and $\lim_{z\to -\frac{1}{\epsilon}}v(z)<\infty$, so we are left to prove the last inequality. 
This follows immediately by observing that in the interval $\left[-\frac{1}{\epsilon}, c \right]$ the functions $\bar{a}^H$ and $\bar{a}^{HH}$ are both regular and bounded away from zero.\\
 \end{proof}

\paragraph{The limit process in the linear case.}
Let us consider the simpler case proposed in the introduction of this section, i.e.,  the case when $\phi_1$ and $\phi_2$ are as in \eqref{Linearphi} and the change of variables $\varphi$ is defined 
 by \eqref{H(x,y) in 0,0} and \eqref{thetaMacro}. 
 Applying the same arguments used in the proof of Proposition \ref{prop:GENH} we are able to obtain an explicit expression for the equation satisfied by the limit process.\\  

We recall that the Legendre's  elliptic integrals of the first and second kind are defined respectively as $F(\varphi | m)= \int_0^\varphi \frac{1}{\sqrt{1-m\sin^2\theta}}\de \theta$ and $E(\varphi | m)=\int_0^\varphi \sqrt{1-m\sin^2\theta}\de \theta$; when $\varphi=\frac{\pi}{2}$ the integrals are called \emph{complete} and we shall denote them respectively by $K(m)$ and $E(m)$. \\
In order to simplify notations, let us pose $R(h,\theta)=\sqrt{1-(1-16h)\sin^2(2\theta)}$ and  $ \beta=2[a(r_2-r_1)+2(b_1+b_2)]$. Then we have 
\begin{eqnarray*}
\bar{a}^H(h)&=&-\frac{(a r_2-a r_1+2b_1+2b_2)}{32K(1-16h)}\int_0^{2\pi}\left(1-\frac{1-R(h,\theta)}{2\sin^2\theta}\right)\left(1-\frac{1-R(h,\theta)}{2\cos^2\theta}\right)\frac{\de \theta}{R(h,\theta)}\\
&=& -\beta h
\end{eqnarray*}

and

\begin{eqnarray*}
\bar{a}^{HH}(h)&=& \frac{\big(ar_2+2b_1\big)}{256 K(1-16h)}\int_0^{2\pi}\left[ \frac{1-R(h,\theta)}{2\sin^2\theta}\left(1-\frac{1-R(h,\theta)}{2\sin^2\theta}\right)\left(1-\frac{1-R(h,\theta)}{2\cos^2\theta}\right)^2\right]\frac{\de \theta}{R(h,\theta)}\\
&+&\frac{\big( -ar_1+2b_2\big)}{256 K(1-16h)} \int_0^{2\pi}\left[ \frac{1-R(h,\theta)}{2\cos^2\theta}\left(1-\frac{1-R(h,\theta)}{2\cos^2\theta}\right)\left(1-\frac{1-R(h,\theta)}{2\sin^2\theta}\right)^2 \right] \frac{\de \theta}{R(h,\theta)}\\
&=& 2(ar_2-ar_1+2b_1+2b_2) h\left[ -\frac{1}{8K(1-16h)}\int_0^{2\pi}\left(  \frac{1-R(h,\theta)}{\cos^2(\theta)} \right)\frac{\de \theta}{R(h,\theta)}
   +(1-16h)\right]\\
   &=&\beta h\left[\frac{E(1-16h)}{K(1-16h)}-16h\right].
\end{eqnarray*} 
Then, adapting the proof of Proposition \ref{existenceuniqueness}, we conclude that the limit process is the (unique, in the sense of probability 
law) weak solution in the interval $I=(0, \frac{1}{16})$ of the equation 
\begin{eqnarray*}
& &\de \bm{H}(t)=  -\beta \bm{H}(t)\de t+\sqrt{\beta \bm{H}(t)\left[\frac{E(1-16\bm{H}(t))}{K(1-16\bm{H}(t))}-16\bm{H}(t)\right]}\de B(t).
\end{eqnarray*}
Letting $\uptau_I=\inf\{t\geq 0: \bm{H}_t\notin (0, \frac{1}{16})\}$, we have \begin{equation*}
P\left(\lim_{t\to \uptau_I} \bm{H}(t\wedge \uptau_{I})=0\right)=P\left(\sup_{0\leq t\leq \uptau_{I}}\bm{H}(t\wedge\uptau_{I}) <\frac{1}{16}\right)=1.
\end{equation*} 
Moreover, \eqref{limitiIepsilon} can be improved by showing that: 
 \begin{equation*}
P(\uptau_I < \infty)=1.
\end{equation*}

Indeed, let us prove that $\lim_{z\to 0^+}v(z)<\infty$. 
The scale function \eqref{scalefunction} is given by  $p(z)= \int_c^z\exp\Big\{16\int_c^u\frac{1}{\frac{E(1-16h)}{K(1-16h)}-16h}\de h\Big\}\de u$.
We again simplify notations by posing $g(h)=\frac{1}{\frac{E(1-16h)}{K(1-16h)}-16h}$, 
so that we can write function $v$ defined in \eqref{vscalefunction} as:
\begin{eqnarray*}
v(z)= \frac{1}{\beta}\int_c^z \left[\exp\Big\{16\int_c^yg(h)\de h\Big\}\int_c^y\exp\Big\{-16\int_c^wg(h)\de h\Big\}\frac{  g(w) }{w }\de w\right]\de y.
\end{eqnarray*}

Note that $g$ is a positive function and, for $c>z>0$, we have $\exp\Big\{16\int_c^xg(h)\de h\Big\}<1$. 
Moreover, by the relations $\lim_{x\to}K[1-x]-\ln(\frac{4}{x})=0$ and $\lim_{x\to0}E[1-16x]=1$ (see \cite{Whittaker} ch.22, p. 521) it follows that: 
\begin{equation*}
\lim_{h\to 0^+}\frac{g(h)}{(-\ln(4h))}=1.
\end{equation*}
Then, $0< -16\int_c^zg(h)\de h \leq 16\int_0^cg(h)\de h\leq C$, where $C$ is a positive constant, and for all $\epsilon>0$ we can choose $c$ sufficiently close to 0 such that: 
\begin{eqnarray*}
v(z) \leq \frac{1}{\beta}\int_z^c \int_y^c e^C\frac{g(w)}{w}\de w \de y \leq -\frac{e^C}{\beta}\int_z^c\int_y^c(1+\epsilon)\frac{\ln(4w)}{w}\de w\de y.
\end{eqnarray*}
From this it follows $\lim_{z\to 0^+}v(z)<\infty$.

\subsection{An averaging principle}\label{averagingpr} 
In this section we prove an Averaging principle for a sequence  $\left\{( X^N, Y^N); N\geq1\right\}$ of Markov processes, where $Y^N$ describes a  ''fast'' variable with values in  $\mathbb{R}/2\pi\mathbb{Z}$ and $X^N$ describes a  ''slow'' variable with values in a closed interval of  $\mathbb{R}$ . This result extends the one of Proposition 3.2 of  \cite{DAIPRA2018} to the case when  the velocity of the fast variable is not necessarily constant.  The idea is to  compare $(X^N,Y^N)$ with a process close to it, where the slow variable is piecewise constant in time.
\begin{theorem}\label{theorem averaging principle}
Let $T>0$ and  $I$ be a closed interval in $\mathbb{R}$. Let $\xi:E=I\times\mathbb{R}/2\pi\mathbb{Z}\to\mathbb{R}$ be a Lipschitz function. Let $\big\{(X^N,Y^N); N\geq 1\big\}$ be a sequence of càdlàg Markov processes, where $\big\{X^N(t),Y^N(t)\big\}_{t\in[0,T]}$  has state space $E^N\subset I\times\mathbb{R}/2\pi\mathbb{Z}$, for all $N$, and denote by $\{\mathcal{F}^N_t\}$ its natural filtration. Let $\gamma>0$ and suppose the following conditions hold:
\begin{itemize}
\item[i)] $\big\{X^N\ ; N\geq 1\big\}$ converges weakly, as $N\to\infty$, to a process $\bar{X}:=\big\{\bar{X}(t)\big\}_{t\in[0,T]}$ with values in $I$ and for all $\zeta>0$ there exists a constant $C_\zeta>0$,    such that for all $\{\mathcal{F}^N_t\}$-stopping time $\tau$ with $\tau\leq T$:
\begin{equation}\label{eq: hyp i)}
E\left[\sup_{t\in[\tau,(\tau+\zeta/N^\gamma)\wedge T]}\left|X^N(t)-X^N(\tau)\right|\right]\leq C_\zeta e(N)
\end{equation}
where $\lim_{N\to\infty}e(N)=0$.

\item[ii)] 
Denoting by $\mathcal{L}_N$ the generator of the process $(X^N,Y^N)$ and by $p_y:I\times\mathbb{R}/2\pi\mathbb{Z}\to\mathbb{R}/2\pi\mathbb{Z}$ the projection on the second coordinate we can write:
\begin{equation*}
\mathcal{L}_Np_y(x^N,y^N)=\left[ N^\gamma\left(F(x^N,y^N)+\delta_N\right)+G(x^N,y^N)\right]\left(1+o(1)\right)
\end{equation*}
where $(x^N,y^N)\in E^N$ and  $\delta_N$ is a sequence converging to zero.  $F$ is  a Lipschitz function in both variables and $\inf_{(x,y)\in E}|F(x,y)|\geq\epsilon>0$; $G$ is a continuous function and $\|G\|_\infty=\sup_{(x,y)\in E}|G(x,y)|<\infty$.

\item[iii)] The martingale given by  $M^N(t) = Y^N(t)-Y^N(0)-\int_0^t\mathcal{L}_Np_y\big(X^N(s),Y^N(s)\big)\de s $ is such that for all $\zeta>0$ there exists a constant  $\bar{C}_\zeta>0$ such that, for all $\{\mathcal{F}^N_t\}$-stopping time $\tau$ with $\tau\leq T$:
\begin{equation}\label{eq: hyp iii)}
E\left[\sup_{t\in [\tau,(\tau+\zeta/N^\gamma)\wedge T]}\left|M^N(t)-M^N(\tau)\right|\right]\leq \bar{C}_\zeta \bar{e}(N)
\end{equation} 
where $\lim_{N\to\infty}\bar{e}(N)=0$.
\end{itemize}
Then, as $N\to\infty$ 
\begin{equation*}
\int_0^T\xi\left(X^N(t),Y^N(t)\right)\de t \xrightarrow{weakly}\int_0^T\bar{\xi}\big(\bar{X}(t)\big)\de t
\end{equation*}
where $\bar{\xi}(x):=\int_0^{2\pi}\xi(x,y)\mu^x(\de y)$ and $\mu^x(\de y)$ is the invariant distribution of the dynamics 
\begin{equation}\label{ODE averaging}
\left\{\begin{array}{ll} 
\dot{X}= 0     \\
\dot{Y}= F(X,Y)
\end{array} \right.
\end{equation}
with $X(0)=x$.
\end{theorem}
\begin{proof}
Arguing as in \cite{DAIPRA2018}, by virtue of the Skorohod representation theorem (see, e.g. \cite{bokar}, Ch.\ 2 Theorem 2.2.2), we can suppose that the processes $\{(X^N,Y^N); N\geq1\}$ are defined on a suitable probability space where $\{X^N;N\geq1\}$ converges to $\bar{X}$ almost surely. We shall prove that on this space we have 
\begin{equation*}
\int_0^T\xi\left(X^N(t),Y^N(t)\right)\de t \xrightarrow{L^1}\int_0^T\bar{\xi}\big(\bar{X}(t)\big)\de t.
\end{equation*}
Let us pose $\mathcal{T}(x):=\int_0^{2\pi}\frac{1}{F(x,y)}\de y$. Observe that, since by assumption $\ \epsilon\leq F\leq \|F\|_{\infty}$, we have, for all $x\in I$, 
\begin{equation}\label{prop uniform bound random time}
\underline{t}:=\frac{2\pi}{\|F\|_\infty}\leq \mathcal{T}(x)\leq \frac{2\pi}{\epsilon}=:\bar{t}\ .
\end{equation}  
 \\
Then the invariant distribution of \eqref{ODE averaging} is given by $\mu^x(\de y)=\frac{1}{\mathcal{T}(x)F(x,y)}\de y$ and  the function $\bar{\xi}$ is Lipschitz. Writing
\begin{multline}\label{ineqTermB}
E\Big[\big|\int_0^T\xi\left(X^N(t),Y^N(t)\right)\de t -\int_0^T\bar{\xi}\big(\bar{X}(t)\big)\de t\big|\Big]\leq\underbrace{E\Big[\big|\int_0^T\xi\left(X^N(t),Y^N(t)\right)\de t -\int_0^T\bar{\xi}\big(X^N(t)\big)\de t\big|\Big]}_B\\
 +E\Big[\big|\int_0^T\bar{\xi}\big(X^N(t)\big)\de t -\int_0^T\bar{\xi}\big(\bar{X}(t)\big)\de t\big|\Big], 
\end{multline}
the last term in the above inequality goes to zero as $N\to\infty$ thanks to the regularity of $\bar{\xi}$, the convergence of $X^N$ to $\bar{X}$ and the dominated convergence theorem. \\
In order to study the term $B$ we introduce a suitable  construction. Fix $\mathfrak{t}_0=0$ and $X^N(0)=x_0,\  Y^N(0)=y_0$ as the initial conditions of 
\begin{equation*}
\left\{\begin{array}{ll}
 \dot{X_1}=0 \\
\dot{Y_1}=NF(X_1,Y_1)
\end{array} \right.
\end{equation*}
and let $Y_1$ be the solution of the above ODE. By the definition of $\mathcal{T}$ we have  $Y_1(\mathcal{T}(x_0)/N)=2\pi+y_0$. \\
Now, for $i\geq 0$, we proceed recursively as follows:
 let $X^N(\mathfrak{t}_i)=x_i$ and $Y^N(\mathfrak{t}_i)= y_i$ be the initial conditions of the equation
\begin{equation*}
\left\{\begin{array}{ll}
 \dot{X}_{i+1}=0\\
\dot{Y}_{i+1}=NF(X_{i+1},Y_{i+1})
\end{array} \right.
\end{equation*}
and denote by $Y_{i+1}$ its solution. Let $\mathcal{T}(x_{i})<\infty$ such that $Y_{i+1}(\mathcal{T}(x_i)/N)=2\pi+y_i$. 
We pose $\mathfrak{t}_{i+1}=\mathfrak{t}_i+\mathcal{T}(x_{i})/N$ and consider the process defined by $(\tilde{X}(t), \tilde{Y}(t)):=(x_i, Y_i(t))$ if $t\in [\mathfrak{t}_i, \mathfrak{t}_{i+1})$.\\
Note that, by \eqref{prop uniform bound random time}, letting $\underline{n}:=\frac{NT}{\bar{t}}$ and $\overline{n}:=\frac{NT}{\underline{t}}$,  it follows that: 
\begin{equation*}
P\big(\underline{n}\leq \big| [0,T]\cap \{\mathfrak{t}_i; i\geq 0\}\big|\leq \overline{n}\big)=1
\end{equation*}
where for a given set $A$, $\ |A|$ denotes its cardinality.

Now, for each $\omega$,  define $\mathfrak{n}(\omega)=\inf\{i: \mathfrak{t}_{i+1}(\omega)>T\}$. Then

\begin{eqnarray*}
\int_0^T \xi\big(X^N(t), Y^N(t)\big) \de t&=& \sum_{i=0}^{\mathfrak{n}-1} \int_{\mathfrak{t}_i}^{\mathfrak{t}_{i+1}} \xi\big(X^N(t), Y^N(t)\big) \de t+ \int_{\mathfrak{t}_{\mathfrak{n}}}^T \xi\big(X^N(t), Y^N(t)\big) \de t\\
&=& 
 \sum_{i=0}^{\overline{n}-1} \int_{\mathfrak{t}_i\wedge T}^{\mathfrak{t}_{i+1}\wedge T} \xi\big(X^N(t), Y^N(t)\big) \de t. 
\end{eqnarray*}

For the term $B$ in \eqref{ineqTermB} it holds:
\begin{multline*}
B \underbrace{\leq \sum_{i=0}^{\overline{n}-1}E\Big[\int_{\mathfrak{t}_i\wedge T}^{\mathfrak{t}_{i+1}\wedge T}\Big|\xi\left(X^N(t),Y^N(t)\right)-\xi\left(X^N(\mathfrak{t}_i),Y^N(t)\right)\Big|\de t\Big]}_{B_1} \\
+  \underbrace{\sum_{i=0}^{\overline{n}-1}E\Big[\int_{\mathfrak{t}_i\wedge T}^{\mathfrak{t}_{i+1}\wedge T}\Big|\xi\left(X^N(\mathfrak{t}_i),Y^N(t)\right)-\xi\left(X^N(\mathfrak{t}_i),Y_i(t)\right)\Big|\de t\Big]}_{B_2}\\ 
+  \underbrace{\sum_{i=0}^{\overline{n}-1}E\Big[\Big|\int_{\mathfrak{t}_i\wedge T}^{\mathfrak{t}_{i+1}\wedge T}\xi\left(X^N(\mathfrak{t}_i),Y_i(t)\right)-\bar{\xi}\left(X^N(t)\right)\de t\Big|\Big]}_{B_3}.
\end{multline*}

We study separately each term of the above  inequality. 
 By hypothesis, the function $\xi$ is Lipschitz (with constant, say, $L_{\xi}$); using \eqref{prop uniform bound random time} and \eqref{eq: hyp i)} we have
\begin{eqnarray*}
B_1&\leq & \sum_{i=0}^{\bar{n}-1}E\Big[ L_{\xi}\int_{\mathfrak{t}_i\wedge T}^{(\mathfrak{t}_{i}+\bar{t}/N)\wedge T}\left|X^N(t)-X^N(\mathfrak{t}_i)\right|\de t\Big] \\ 
& \leq &  
\sum_{i=0}^{\bar{n}-1}L_\xi\frac{\bar{t}}{N}E\Big[\sup_{t\in[\mathfrak{t}_i\wedge T,(\mathfrak{t}_i+\bar{t}/N)\wedge T]}\left|X^N(t)-X^N(\mathfrak{t}_i)\right|\Big]\leq \overline{n}L_{\xi}\frac{\bar{t}}{N}C_{\bar{t}}e(N)\xrightarrow{N\to \infty} 0.
\end{eqnarray*}

Analogously, for the term $B_2$  we obtain
\begin{equation*}
B_2\leq \sum_{i=0}^{\bar{n}-1}L_\xi\frac{\bar{t}}{N}E\Big[\sup_{t\in[\mathfrak{t}_i\wedge T,(\mathfrak{t}_i+\bar{t}/N)\wedge T]}\left|Y^N(t)-Y_i(t)\right|\Big].
\end{equation*}
By hypothesis $ii)$ and by the construction of $Y_i$  we can write:
\begin{eqnarray}
Y^N(t\wedge T)-Y_i(t \wedge T)  =  N\int_{\mathfrak{t}_i\wedge T}^{t\wedge T}F\big(X^N(s),Y^N(s)\big)-F\big(X^N(\mathfrak{t}_i),Y_i(s)\big) \de s  +N\int_{\mathfrak{t}_i\wedge T}^{t\wedge T}\delta_N\de s
\nonumber \\
 + \int_{\mathfrak{t}_i\wedge T}^{t\wedge T}G\big(X^N(s),Y^N(s)\big)\de s + M^N(t\wedge T)- M^N(\mathfrak{t}_i\wedge T) + o(1). \label{eq: Y-Y_i} 
\end{eqnarray}
The function $F$ is Lipschitz in both variables with constant, say $L_F$, then from \eqref{eq: Y-Y_i}  we obtain
\begin{eqnarray*}
\sup_{t\in[\mathfrak{t}_i\wedge T,(\mathfrak{t}_i+\bar{t}/N)\wedge T]}\left|Y^N(t)-Y_i(t)\right| &\leq &  NL_F\int_{\mathfrak{t}_i\wedge T}^{(\mathfrak{t}_i+\bar{t}/N)\wedge T}\sup_{h\in[\mathfrak{t}_i\wedge T,s\wedge T]}\left|X^N(h)-X^N(\mathfrak{t}_i)\right|\de s  \nonumber\\
& + &  NL_F \int_{\mathfrak{t}_i\wedge T}^{(\mathfrak{t}_i+\bar{t}/N)\wedge T}\sup_{h\in[\mathfrak{t}_i\wedge T,s\wedge T]}\left|Y^N(h)-Y_i(h)\right|\de s   \nonumber \\
 &+ &\sup_{t\in[\mathfrak{t}_i\wedge T,(\mathfrak{t}_i+\bar{t}/N)\wedge T]}|M^N(t)- M^N(\mathfrak{t}_i)| \\
 &+& \bar{t}\delta_N+\frac{\bar{t}}{N}\|G\|_\infty +o(1).
\end{eqnarray*}
Define $f^i(s):=\left[\sup_{t\in[\mathfrak{t}_i \wedge T,s\wedge T]}\left|Y^N(t)-Y_i(t)\right|\right]$ then 
\begin{eqnarray*}
f^i(\mathfrak{t}_i+\bar{t}/N) & \leq & NL_F\int_{\mathfrak{t}_i\wedge T}^{(\mathfrak{t}_i+\bar{t}/N)\wedge T}f^i(s)\de s \\
& & + NL_F \int_{\mathfrak{t}_i\wedge T}^{(\mathfrak{t}_i+\bar{t}/N)\wedge T}\sup_{h\in[\mathfrak{t}_i\wedge T,s\wedge T]}\left|X^N(h)-X^N(\mathfrak{t}_i)\right|\de s \\
& & +\bar{t}\delta_N+ \frac{\bar{t}}{N}\|G\|_\infty + \sup_{t\in[\mathfrak{t}_i\wedge T,(\mathfrak{t}_i+\bar{\tau}/N)\wedge T]}\left|M^N(t)- M^N(\mathfrak{t}_i)\right| +o(1) \\
& & = NL_F\int_{\mathfrak{t}_i\wedge T}^{(\mathfrak{t}_i+\bar{t}/N)\wedge T}f^i(s)\de s + R_i\big((t_i+\bar{t}/N)\wedge T\big). 
\end{eqnarray*}

By Gronwall inequality we obtain
\begin{equation*}
f^i(\mathfrak{t}_i+\bar{t}/N) \leq e^{L_F\bar{t}}R_i\big((t_i+\bar{t}/N)\wedge T\big).
\end{equation*}
Then, taking expectations and using \eqref{eq: hyp i)} and \eqref{eq: hyp iii)} we have

\begin{equation*}
B_2\leq L_\xi\frac{\overline{n}\bar{t}}{N}e^{L_F\bar{t}}\Big(L_F\bar{t}C_{\bar{t}}e(N)+\bar{t}\delta_N+\frac{\bar{t}}{N}\|G\|+\bar{C}_{\bar{t}}\bar{e}(N)	\Big)\xrightarrow{N\to\infty}0.
\end{equation*} 
Recalling the construction of the process $\tilde{Y}$, we change variable in each integral of the sum in $B_3$, setting $\theta=Y_i(t)$. Note that if $\mathfrak{t}_{i+1}<T$ using the  periodicity of $F$ and  $\xi$ we can write each integral as  
\begin{eqnarray}
\int_{\mathfrak{t}_i}^{\mathfrak{t}_{i+1}}\xi\left(X^N(\mathfrak{t}_i),Y_i(t)\right)\de t & = &\int_{Y_i(\mathfrak{t}_i)}^{Y_i(\mathfrak{t}_{i+1})}\xi\left(X^N(\mathfrak{t}_i),\theta\right)\frac{\de\theta}{NF\left(X^N(\mathfrak{t}_i),\theta\right)}  \nonumber \\
& = &  \int_{0}^{2\pi}\xi\big(X^N(\mathfrak{t}_i),\theta\big)\frac{\de\theta}{NF\big(X^N(\mathfrak{t}_i),\theta\big)}.  \nonumber 
\end{eqnarray}

Using the definition of the invariant measure $\mu$ we have
\begin{multline*}
B_3\leq \sum_{i=0}^{\overline{n}-1}E\Big[\Big|(\mathfrak{t}_{i+1}\wedge T-\mathfrak{t}_i \wedge T)\int_0^{2\pi}\xi\left(X^N(\mathfrak{t}_i),\theta\right)\mu^{X^N(\mathfrak{t}_i)}(\de\theta)-\int_{\mathfrak{t}_i \wedge T}^{\mathfrak{t}_{i+1}\wedge T}\int_0^{2\pi}\xi\left(X^N(t),\theta\right)\mu^{X^N(t)}(\de\theta)\de t\Big|\Big] \\
+\sum_{i=0}^{\overline{n}-1} E\Big[\mathbbm{1}_{\mathfrak{t}_i\leq T<\mathfrak{t}_{i+1}}\int_{\mathfrak{t}_i}^{T}\big|\xi\left(X^N(\mathfrak{t}_i),Y_i(t)\right)\big|\de t  \Big].
\end{multline*}

Note that the only non-zero term of the last sum is the one corresponding to $i=\mathfrak{n}$ and that $\frac{\underline{t}}{N} \leq |T-\mathfrak{t}_{\mathfrak{n}}|\leq \frac{\bar{t}}{N}$ for all $N$ . Therefore, recalling that $\xi$ is uniformly bounded on $E$   the last term of the inequality tends to zero as $N\to\infty$. \\

For the first sum, arguing as for the term $B_1$ and  applying again \eqref{eq: hyp i)} we obtain 
\begin{equation*}
B_3 \leq \sum_{i=0}^{\overline{n}-1} E\Big[\Big|(\mathfrak{t}_{i+1}\wedge T-\mathfrak{t}_i \wedge T)\bar{\xi}\left(X^N(\mathfrak{t}_i)\right)-\int_{\mathfrak{t}_i \wedge T}^{\mathfrak{t}_{i+1} \wedge T}\bar{\xi}\big(X^N(t)\big)\de t\Big|\Big]\leq K e(N)\xrightarrow{N\to\infty}0
\end{equation*}
where $K$ is a suitable constant.
\end{proof}

\subsection{Proof of Theorem \ref{Main}}\label{proofMain} 

In this subsection we use the following notation:
\begin{equation*}
\Big\{\big\{\big(\tilde{\bm{H}}^N(t),\tilde{\bm{\Theta}}^N(t)\big)\big\}_{t\in [0,T]}\ ; N\geq1 \Big\} := \Big\{ \big\{\big(\bm{H}^N(Nt), \bm{\Theta}^N(Nt)\big)\big\}_{t\in[0,T]}\ ;N\geq1 \Big\}.
\end{equation*} 

The main tool required for the proof of Theorem \ref{Main} is the following  proposition:
\begin{prop}\label{prop same distribution}
For all $\epsilon,\epsilon'>0$ and for any initial condition $h_0\in(-\frac{1}{\epsilon}+\delta,-\epsilon'-\delta)$,  the sequence $\{\tilde{\bm{H}}^N(\cdot\wedge\uptau_{-\frac{1}{\epsilon},-\epsilon'}^N);N\geq1\}$ weakly converges (up to passing to a subsequence) , as $N\to\infty$, to a continuous process  $\tilde{\bm{H}}_{-\frac{1}{\epsilon},-\epsilon'}$. Moreover, let $\bm{H}$ be the unique solution of (\ref{equation SDE}) with $\bm{H}(0)=h_0$. For a fixed $\delta>0$, define
\begin{equation*}
\tilde{\uptau}_{-\frac{1}{\epsilon}+\delta,-\epsilon'-\delta}=\inf\left\{t\in[0,T] : \tilde{\bm{H}}_{-\frac{1}{\epsilon},-\epsilon'}(t)\notin(-\frac{1}{\epsilon}+\delta,-\epsilon'-\delta) \right\}
\end{equation*}
and 
\begin{equation*}
\uptau_{-\frac{1}{\epsilon}+\delta,-\epsilon'-\delta}=\inf\left\{t\in[0,T] : \bm{H}(t)\notin(-\frac{1}{\epsilon}+\delta,-\epsilon'-\delta)\right\}.
\end{equation*}
Then, $\tilde{\bm{H}}_{-\frac{1}{\epsilon},-\epsilon'}(\cdot\wedge\tilde{\uptau}_{-\frac{1}{\epsilon}+\delta,-\epsilon'-\delta}) $ and $ \bm{H}(\cdot\wedge \uptau_{-\frac{1}{\epsilon}+\delta,-\epsilon'-\delta})$ have the same distribution.
\end{prop}

The proof of this proposition needs some preliminary results.

\paragraph{Tightness.}  Let $T>0$ be fixed and let $\epsilon,\epsilon'>0$ be small constants. In order to prove the tightness for the sequence of stopped processes $\big\{ \tilde{\bm{H}}^N(\cdot\wedge\uptau_{-\frac{1}{\epsilon},-\epsilon'}^N); N\geq1 \big\}$ we  use the \emph{Aldous' tightness criterion} (see \cite{comets1988}), namely, we check the following sufficient conditions:

\begin{itemize}
\item[i)] for every $\varepsilon>0$ there exist a constant $C>0$ such that 
\begin{equation*}
\sup_N P\left(\sup_{t\in[0,T]} \left|\tilde{\bm{H}}^N(t\wedge\uptau_{-\frac{1}{\epsilon},-\epsilon'}^N)\right|\geq M  \right)\leq \varepsilon;
\end{equation*}
\item[ii)] for any $\varepsilon>0$ and $\alpha>0$ there exists $\delta>0$ such that
\begin{equation*}
\sup_N \sup_{\substack{0\leq\uptau_1\leq\uptau_2\leq (\uptau_1+\delta)\wedge T \\ \uptau_1,\uptau_2\ stopping\ times} } P\left( \left|\tilde{\bm{H}}^N(\uptau_2\wedge\uptau_{-\frac{1}{\epsilon},-\epsilon'}^N)-\tilde{\bm{H}}^N(\uptau_1\wedge\uptau_{-\frac{1}{\epsilon},-\epsilon'}^N)\right|\geq\alpha\right) \leq \varepsilon.
\end{equation*}
\end{itemize}

\begin{prop}\label{prop tightness H^N}
For any $T>0$, the sequence $\Big\{ \{\tilde{\bm{H}}^N(t\wedge\uptau^N_{-\frac{1}{\epsilon},-\epsilon'}) \}_{t\in [0,T]}\ ; N>1\Big\}$ is tight. 
\end{prop}
\begin{proof}
 As in Proposition \ref{prop uniform estimate H}, we can write
\begin{eqnarray}\label{esplicit expression H^N}
\tilde{\bm{H}}^N\big(s\wedge\uptau_{-\frac{1}{\epsilon},-\epsilon'}^N\big)- \tilde{\bm{H}}^N(0) &=& \int_{0}^{s\wedge\uptau_{-\frac{1}{\epsilon},-\epsilon'}^N} \tilde{\mathcal{K}}_{N}^{\epsilon,\epsilon'} p_x\left(\tilde{\bm{H}}^N(u),\tilde{\bm{\Theta}}^N(u)\right)\de u\nonumber\\& &+  \tilde{\bm{M}}^N(s\wedge\uptau_{-\frac{1}{\epsilon},-\epsilon'}^N), 
\end{eqnarray}
so that condition $i)$ of the \emph{Aldous' criterion} is immediately satisfied since $\left|\tilde{\bm{H}}^N(t\wedge\uptau_{-\frac{1}{\epsilon},-\epsilon'}^N) \right| \leq \frac{1}{\epsilon}$.\\ 
In order to check condition $ii)$, let us fix $\varepsilon, \alpha>0$ and take any pair of stopping times $\uptau_1,\uptau_2$ with $0\leq \uptau_1\leq\uptau_2\leq(\uptau_1+\delta)\wedge T$ for some $\delta$.  By \eqref{esplicit expression H^N} we have 
\begin{multline}
\left|\tilde{\bm{H}}^N(\uptau_2\wedge\uptau_{-\frac{1}{\epsilon},-\epsilon'}^N)- \tilde{\bm{H}}^N(\uptau_1\wedge\uptau_{-\frac{1}{\epsilon},-\epsilon'}^N)\right |= \Big|\int_{\uptau_1\wedge\uptau_{-\frac{1}{\epsilon},-\epsilon'}^N}^{\uptau_2\wedge\uptau_{-\frac{1}{\epsilon},-\epsilon'}^N} \tilde{\mathcal{K}}_{N}^{\epsilon,\epsilon'}p_x\left( \tilde{\bm{H}}^N(u),\tilde{\bm{\Theta}}^N(u)\right)\de u \\
 +  \tilde{\bm{M}}^{N,\uptau_1\wedge\uptau_{-\frac{1}{\epsilon},-\epsilon'}^N}(\uptau_2\wedge\uptau_{-\frac{1}{\epsilon},-\epsilon'}^N)\Big| \nonumber
\end{multline}
where $ \tilde{\bm{M}}^{N,\uptau_1\wedge\uptau_{-\frac{1}{\epsilon},-\epsilon'}^N}(\uptau_2\wedge\uptau_{-\frac{1}{\epsilon},-\epsilon'}^N) := \tilde{\bm{M}}^N(\uptau_2\wedge\uptau_{-\frac{1}{\epsilon},-\epsilon'}^N) -\tilde{\bm{M}}^N(\uptau_1\wedge\uptau_{-\frac{1}{\epsilon},-\epsilon'}^N)  $.  Using the optional stopping theorem and the Ito isometry, we have
\begin{eqnarray*}
E\Big[ \Big( \tilde{\bm{M}}^{N,\uptau_1\wedge\uptau_{-\frac{1}{\epsilon},-\epsilon'}^N}(\uptau_2\wedge\uptau_{-\frac{1}{\epsilon},-\epsilon'}^N)\Big)^2\Big] = E\Big[  \Big( \tilde{\bm{M}}^N(\uptau_2\wedge\uptau_{-\frac{1}{\epsilon},-\epsilon'}^N)\Big)^2- \Big( \tilde{\bm{M}}^N(\uptau_1\wedge\uptau_{-\frac{1}{\epsilon},-\epsilon'}^N)\Big)^2\Big] & \\
=  E\Big[ \int_{\uptau_1\wedge\uptau_{-\frac{1}{\epsilon},-\epsilon'}^N}^{\uptau_2\wedge\uptau_{-\frac{1}{\epsilon},-\epsilon'}^N}  \sum_{k=1,2}\sum_{j=1}^{N_k}(\Delta_{i,k}\tilde{\bm{H}}^N)^2(s-)\tilde{\bm{\lambda}}^N_{i,k}(s-)\de s  \Big] 
\leq  C_2 (\uptau_1\wedge\uptau_{-\frac{1}{\epsilon},-\epsilon'}^N -\uptau_2\wedge\uptau_{-\frac{1}{\epsilon},-\epsilon'}^N) & \leq C_2 \delta
\end{eqnarray*}
where $\left(\Delta_{i,k}\tilde{\bm{H}}^N\right)(s-):=\tilde{\bm{H}}^N(s)-\tilde{\bm{H}}^N(s-)$ denotes the jump amplitude of $\tilde{\bm{H}}^N$ corresponding to a jump of the component $\bm{\sigma}^N_{i,k}$ at time $s$ and, by an easy computation,  yields  $(\Delta_{i,k}\tilde{\bm{H}}^N)^2(s-)\leq C_2\frac{1}{N^2}$, for a constant $C_2>0$. Then by Chebychev inequality 
\begin{eqnarray*}
P\left(\Big|\tilde{\bm{M}}^{N,\uptau_1\wedge\uptau_{-\frac{1}{\epsilon},-\epsilon'}^N}(\uptau_2\wedge\uptau_{-\frac{1}{\epsilon},-\epsilon'}^N)\Big|\geq \alpha\right)\leq \frac{E\left[ \left( \tilde{\bm{M}}^{N,\uptau_1\wedge\uptau_{-\frac{1}{\epsilon},-\epsilon'}^N}(\uptau_2\wedge\uptau_{-\frac{1}{\epsilon},-\epsilon'}^N)\right)^2\right]}{\alpha^2}\leq\frac{C_2\delta}{\alpha^2}.
\end{eqnarray*}
Choosing $\delta$ sufficiently small the proposition holds true.
\end{proof}

\paragraph{Averaging principle for the stopped processes.} In this paragraph we show that the sequence $\big\{\big(\tilde{\bm{H}}^N(\cdot\wedge\uptau_{-\frac{1}{\epsilon},-\epsilon'}^N),\tilde{\bm{\Theta}}^N(\cdot\wedge\uptau_{-\frac{1}{\epsilon},-\epsilon'}^N)\big); N>1\big\}$ satisfies the conditions of Theorem \ref{theorem averaging principle}. 
We start by proving the following
\begin{lemma} \label{prop uniform estimate H}
For any bounded $\{\mathcal{F}_t\}$-stopping time $\tau$ and any $\zeta\in \mathbb{R}^+$ there exists a constant $C_\zeta$, independent of $N$ and $\tau$, such that:
\begin{equation*}
E\left[\sup_{s\in[\tau,\tau+\zeta/N]} \big|\tilde{\bm{H}}^N(s\wedge\uptau_{-\frac{1}{\epsilon},-\epsilon'}^N)-\tilde{\bm{H}}^N(t\wedge\uptau_{-\frac{1}{\epsilon},-\epsilon'}^N)\big| \right] \leq \frac{C_\zeta}{\sqrt{N}}.
\end{equation*}
\end{lemma}
\begin{proof}
From the definition of the process $\tilde{\bm{H}}^N$ we can write:
\begin{equation}\label{eq: dH N}
d\tilde{\bm{H}}^N(s)=\int_0^\infty \sum_{k=1,2}\sum_{j=1}^{N_k}\left(\Delta_{i,k}\tilde{\bm{H}}^N\right)(s-)\mathbbm{1}_{\left(0,\tilde{\lambda}^N\left(\bm{\sigma}_{i,k}^N(s-),\bm{m}_k^N(s-),\bm{m}_{k'}^N(s-)\right)\right]}(u)\mathcal{N}^{i,k}(\de s,\de u)
\end{equation}
where $\tilde{\lambda}^N(\sigma_{i,k}^N,m_k^N,m_{k'}^N) = N\lambda^N(\sigma_{i,k}^N,m_k^N,m_{k'}^N) $ with $\lambda^N$ being the jump rate function defined in (\ref{rate function lambdaN}). In what follows we use the short notation:
 $$\tilde{\bm{\lambda}}_{i,k}^N(s-):=\tilde{\lambda}^N\big(\bm{\sigma}_{i,k}^N(s-),\bm{m}_k^N(s-),\bm{m}_{k'}^N(s-)\big).$$ 

Let $\tau$ be a bounded stopping time and $s\geq 0$. Then 
\begin{eqnarray}
\tilde{\bm{H}}^N\left((\tau+s)\wedge\uptau_{-\frac{1}{\epsilon},-\epsilon'}^N\right)- \tilde{\bm{H}}^N\left(\tau\wedge\uptau_{-\frac{1}{\epsilon},-\epsilon'}^N\right) =&&  \int_{ \tau\wedge\uptau_{-\frac{1}{\epsilon},-\epsilon'}^N}^{(\tau+s)\wedge\uptau_{-\frac{1}{\epsilon},-\epsilon'}^N} \tilde{\mathcal{K}}_{N}^{\epsilon,\epsilon'}p_x\left( \tilde{\bm{H}}^N(u),\tilde{\bm{\Theta}}^N(u)\right)\de u \nonumber \\ 
&& + \tilde{\bm{M}}^N\left((\tau+s)\wedge\uptau_{-\frac{1}{\epsilon},-\epsilon'}^N\right)-
 \tilde{\bm{M}}^N\left(\tau\wedge\uptau_{-\frac{1}{\epsilon},-\epsilon'}^N\right) 
 \nonumber
\end{eqnarray}
where $p_x$ is the projection on the first coordinate and the martingale $\tilde{\bm{M}}^N$ is obtained by the sum \eqref{eq: dH N} by replacing $\mathcal{N}^{i,k}$ with its compensated process $\tilde{\mathcal{N}}^{i,k}$  for each $k,i$. 
Now, applying the Burkholder-Davis-Gundy  inequality to the martingale 
$\{\tilde{\bm{R}}^N(s)\}_{s\geq 0}$ defined by $\tilde{\bm{R}}^N(s):=\tilde{\bm{M}}^N\big((\tau+s)\wedge\uptau_{-\frac{1}{\epsilon},-\epsilon'}^N\big)-
 \tilde{\bm{M}}^N\big(\tau\wedge\uptau_{-\frac{1}{\epsilon},-\epsilon'}^N\big)$ we get: 
 \begin{equation*}
E\Big[  \sup_{s\in[0,\zeta/N]} \big| \tilde{\bm{R}}^N(s)\big|\Big] \leq C_1 E\Big[\big\langle \tilde{\bm{R}}^N\big\rangle_{\zeta/N}^{\frac{1}{2}}\Big]
\end{equation*}
where $C_1$ is a constant and $\langle \tilde{\bm{R}^N}\rangle$ denotes the quadratic variation of $\tilde{\bm{R}}^N$. 
Using the fact that $\tilde{\bm{M}}^N$ is the sum of  orthogonal martingales we obtain: 
\begin{eqnarray*}
  E\Big[\big\langle \tilde{\bm{R}}^N\big\rangle_{\zeta/N}^{\frac{1}{2}}\Big]&=&
 E\Big[\Big(\int_{\tau\wedge\uptau_{-\frac{1}{\epsilon},-\epsilon'}^N}^{(\tau+\zeta/N)\wedge\uptau_{-\frac{1}{\epsilon},-\epsilon'}^N} \int_0^\infty \sum_{k=1,2}\sum_{j=1}^{N_k}(\Delta_{i,k}\tilde{\bm{H}}^N)^2(s-)\mathbbm{1}_{\left(0,\tilde{\bm{\lambda}}_{i,k}^N(s-)\right]}(u)\hat{\mathcal{N}}^{i,k}(\de s,\de u)\Big)^{\frac{1}{2}} \Big]\\
&=&E\Big[\Big(\int_{\tau\wedge\uptau_{-\frac{1}{\epsilon},-\epsilon'}^N}^{(\tau+\zeta/N)\wedge\uptau_{-\frac{1}{\epsilon},-\epsilon'}^N}  \sum_{k=1,2}\sum_{j=1}^{N_k}(\Delta_{i,k}\tilde{\bm{H}}^N)^2(s-)\tilde{\bm{\lambda}}_{i,k}^N(s-)\de s \Big)^{\frac{1}{2}} \Big]
\end{eqnarray*}
where $\hat{\mathcal{N}}^{i,k}$ is the compensator of $\mathcal{N}^{i,k}$. We recall that $(\Delta_{i,k}\tilde{\bm{H}}^N)^2(s-)\leq C_2\frac{1}{N^2}$ and we note also that the jump rate function  satisfies $\|\tilde{\lambda}^N\|_\infty\leq C_3 N$ for a constant $C_3>0$. 
Then choosing the right constant $K_\zeta$ we obtain:
\begin{equation}\label{DIS}
E\Big[  \sup_{s\in[0,\zeta/N]} \big| \tilde{\bm{R}}^N(s)\big|\Big]  \leq \frac{K_\zeta}{\sqrt{N}}.
\end{equation}
Moreover, from the properties of the generator $\mathcal{K}_{N}^{\epsilon,\epsilon'}$ we know that, for any $C^2$-function $f$, the function $\left(\mathcal{K}_{N}^{\epsilon,\epsilon'}(f\circ p_x)\right)$ is uniformly bounded. Then there exists a constant $C>0$, independent of $N$ such that
\begin{equation*}
E\Big[\sup_{s\in[\tau, \tau+\zeta/N]}\big|
 \int_{ \tau\wedge\uptau_{-\frac{1}{\epsilon},-\epsilon'}^N}^{(\tau+s)\wedge\uptau_{-\frac{1}{\epsilon},-\epsilon'}^N} \tilde{\mathcal{K}}_{N}^{\epsilon,\epsilon'}p_x\left( \tilde{\bm{H}}^N(u),\tilde{\bm{\Theta}}^N(u)\right)\de u\big|\Big]\leq \frac{C}{N}
\end{equation*}
and the proof is complete.
\end{proof}
Let $\mathcal{A}$ be the operator defined on functions $f\in\mathcal{C}^2\big((-\infty,0)\times\mathbb{R}/2\pi\mathbb{Z}\big)$ by:
\begin{equation} \label{def: gen A}
\mathcal{A}f(h,\theta):= a^H(h,\theta)f_{H}(h,\theta)+a^{HH}(h,\theta)f_{HH}(h,\theta)
\end{equation}
where $a^{H},a^{HH}$ are defined as in \eqref{generator micro001}. 

\begin{prop} \label{Averaging principle}
Consider the sequence $\big\{\big(\tilde{\bm{H}}^N(\cdot\wedge\uptau_{-\frac{1}{\epsilon},-\epsilon'}^N),\tilde{\bm{\Theta}}^N(\cdot\wedge\uptau_{-\frac{1}{\epsilon},-\epsilon'}^N)\big); N>1\big\}$ and the weak limit $\tilde{\bm{H}}_{-\frac{1}{\epsilon},-\epsilon'}$ of $\{\tilde{\bm{H}}^N(\cdot\wedge\uptau_{-\frac{1}{\epsilon},-\epsilon'}^N) ; N>1\}$.
For any $f\in \mathcal{C}^3([-\frac{1}{\epsilon},-\epsilon'])$, up to passing to a subsequence, we have, as $N\to \infty$:
\begin{equation*}
\int_0^T \mathcal{A}f\left(\tilde{\bm{H}}^N(t\wedge\uptau_{-\frac{1}{\epsilon},-\epsilon'}^N),\bm{\Theta}(t\wedge\uptau_{-\frac{1}{\epsilon},-\epsilon'}^N)\right)\de t \xrightarrow{weakly} \int_0^T \bar{\mathcal{A}}f\left(\tilde{\bm{H}}_{-\frac{1}{\epsilon},-\epsilon'}(t) \right)\de t
\end{equation*}
where 
\begin{equation*}
\bar{\mathcal{A}}f( h ) := \bar{a}^H(h)f_H(h) + \bar{a}^{HH}(h)f_{HH}(h)
\end{equation*}
and $\bar{a}^H$, $\bar{a}^{HH}$ are defined as in \eqref{coefficientsSDEa^H} and \eqref{coefficientsSDEa^HH}.
\end{prop}	
\begin{proof}
We first observe that by the regularity of  $a^H$ and $a^{HH}$ on $[-\frac{1}{\epsilon},-\epsilon']\times\mathbb{R}/2\pi\mathbb{Z}$ it follows that, for any $f\in C^3([-\frac{1}{\epsilon},-\epsilon'])$, the function $\mathcal{A}f:[-\frac{1}{\epsilon},-\epsilon']\times \mathbb{R}/2\pi\mathbb{Z}\to \mathbb{R}$ is Lipschitz in both variables.  
By Proposition \ref{prop tightness H^N} and  Lemma \ref{prop uniform estimate H} hypothesis $i)$ of Theorem \ref{theorem averaging principle} is satisfied. 

From \eqref{generator micro001} and the fact that  $F\geq c(\epsilon,\epsilon')>0$ (see the proof of Proposition \ref{prop:GENH} in subsection \ref{s:mainresult})  we immediately observe that hypothesis $ii)$ holds too. To show hypothesis $iii)$ we find a uniform bound for the martingale term in the relation:

\begin{eqnarray*}\label{eq: integ Theta^N}
\tilde{\bm{\Theta}}^N(s\wedge\uptau_{-\frac{1}{\epsilon},-\epsilon'}^N)-\tilde{\bm{\Theta}}^N(t\wedge\uptau_{-\frac{1}{\epsilon},-\epsilon'}^N)=&& \int_{t\wedge\uptau_{-\frac{1}{\epsilon},-\epsilon'}^N}^{s\wedge\uptau_{-\frac{1}{\epsilon},-\epsilon'}^N}\tilde{\mathcal{K}}_{N}^{\epsilon,\epsilon'}p_y\left(\tilde{\bm{H}}^N(u),\tilde{\bm{\Theta}}^N(u)\right)\de u\\
&& + \bm{M}^N(s\wedge\uptau_{-\frac{1}{\epsilon},-\epsilon'}^N)-\bm{M}^N(t\wedge\uptau_{-\frac{1}{\epsilon},-\epsilon'}^N)
\end{eqnarray*}
where $t,s\in[0,T]$ and $p_y$ denotes the projection on the second coordinate.

From \eqref{eq: integ Theta^N} we write 
\begin{equation*}
\de\tilde{\bm{\Theta}}^N(s)=\int_0^\infty \sum_{k=1,2}\sum_{j=1}^{N_k}\big(\Delta_{j,k}\tilde{\bm{\Theta}}^N\big)(s-) \mathbbm{1}_{\left(0,\tilde{\bm{\lambda}}^N_{i,k}(s-)\right]}(u)\mathcal{N}^{j,k}(\de s,\de u)
\end{equation*}

where  $\big(\Delta_{j,k}\tilde{\bm{\Theta}}^N\big)(s-)$ denotes the jump amplitude of $\tilde{\bm{\Theta}}^N$ corresponding to a jump of the component $\bm{\sigma}_{i,k}^N$ at time $s$. If we write it explicitly using the change of variables \eqref{thetaMacro} we immediately obtain that $|\Delta_{j,k}\tilde{\bm{\Theta}}^N|\leq C\frac{1}{N} + o(\frac{1}{N^2})$ for a suitable constant $C>0$. 
By the same argument used to obtain \eqref{DIS} in the proof of Lemma \ref{prop uniform estimate H} we get, for all stopping time $\tau$:
\begin{equation*}
E\left[\sup_{s\in[\tau,\tau+\zeta/N]} \left|\bm{M}^N(s\wedge\uptau_{-\frac{1}{\epsilon},-\epsilon'}^N)-\bm{M}^N(t\wedge\uptau_{-\frac{1}{\epsilon},-\epsilon'}^N) \right|\right] \leq  \frac{C_\zeta}{\sqrt{N}}. 
\end{equation*}

\end{proof}

\paragraph{Proof of Proposition \ref{prop same distribution}} 
Before giving the proof we state a useful result. 
Let $L$ be a linear operator defined for bounded measurable functions on a metric space $E$, let $U$  be  an open subset of $E$ and  let $X$ be a càdlàg process. We recall that the process $X(\cdot\wedge\tau)$, where $\tau$ is the exit time from $U$ of the process $X$,  is said to be a solution of the $(L,U)$-stopped martingale problem if
\begin{equation}\label{stoppedmart}
f\big(X(t\wedge\tau)\big)-f\big(X(0)\big)-\int_0^{t\wedge\tau}Lf\big(X(s)\big)\de s
\end{equation}
is a martingale for all $f\in dom(L)$.
\begin{theorem} \label{teo:Kurtz stopping}
(\cite{EthierKurtz}, Ch.\ 4 Thm 6.1) Let $(E,d)$ be a Polish space and let $L$ be a linear operator $L:\mathcal{C}_b(E)\to B(E)$. If the $\mathcal{D}([0,T],E)$ martingale problem for $L$ is well-posed , then for any open set $U\subset E$ there exists a unique solution of the stopped martingale problem $(L,U)$. 
\end{theorem}

\noindent We will show that  \eqref{stoppedmart} holds for $X=\tilde{\bm{H}}_{-\frac{1}{\epsilon}, -\epsilon}$ and $L=\bar{\mathcal{A}}$ with $f$ in a subset of $dom(\bar{\mathcal{A}})$ which is \emph{not} a measure determining class for $\mathcal{D}\big([0,T];[-\frac{1}{\epsilon},-\epsilon']\big)$ . This motivates the restriction to the interval $(-\frac{1}{\epsilon}+\delta. -\epsilon'-\delta)$ in the statement of Proposition \ref{prop same distribution}.

\begin{proof}[proof of Proposition \ref{prop same distribution}]
Weak convergence of $\big\{\tilde{\bm{H}}(\cdot\wedge\uptau_{-\frac{1}{\epsilon},-\epsilon'}^N);N\geq1\big\}$ to  $\tilde{\bm{H}}_{-\frac{1}{\epsilon},-\epsilon'}$ follows immediately from Proposition \ref{prop tightness H^N}.  Notice that $\sup_{s\in [0,T\wedge \uptau^N_{-\frac{1}{\epsilon}, -\epsilon'}]}|\tilde{\bm{H}}^N(s)-\tilde{\bm{H}}^N(s-)|\leq \frac{C}{N}$ for a suitable constant $C$, then (see \cite{EthierKurtz} Ch.\ 3, Thm 10.2) the limit process  $\tilde{\bm{H}}_{-\frac{1}{\epsilon},-\epsilon'}$ is continuous.\\
In the setting of the proof of Proposition \ref{averagingpr} we consider a suitable probability space where the  above convergence is almost sure. In such space we have also, up to passing to a subsequence,  
\begin{equation}\label{eq: converg as averag}
\int_{s}^{t}\mathcal{A}f\left(\tilde{\bm{H}}^N(u\wedge\uptau_{-\frac{1}{\epsilon},-\epsilon'}^N),\tilde{\bm{\Theta}}^N(u\wedge\uptau_{-\frac{1}{\epsilon},-\epsilon'}^N )\right)\de u  \xrightarrow{a.s.}  \int_s^t \bar{\mathcal{A}}f \left(\tilde{\bm{H}}_{-\frac{1}{\epsilon},-\epsilon'}(u)\right)\de u.
\end{equation}
Now for any $f\in C^3_c\big((-\frac{1}{\epsilon},-\epsilon')\big)$ the process defined by
\begin{equation*}
\bm{Z}_{-\frac{1}{\epsilon},-\epsilon'}^f(t):=f\left(\tilde{\bm{H}}_{-\frac{1}{\epsilon},-\epsilon'}(t)\right)-f\left(\tilde{\bm{H}}_{-\frac{1}{\epsilon},-\epsilon'}(0)\right)-\int_0^t\bar{\mathcal{A}}f\left(\tilde{\bm{H}}_{-\frac{1}{\epsilon},-\epsilon'}(s)\right)\de s
\end{equation*}
is a martingale. Indeed consider  the martingale 
\begin{equation*}
\bm{M}^{N,f}(t):= f\left(\tilde{\bm{H}}^N(t\wedge\uptau_{-\frac{1}{\epsilon},-\epsilon'}^N)\right)-f\left(\tilde{\bm{H}}^N(0)\right)-\int_0^{t\wedge\uptau_{-\frac{1}{\epsilon},-\epsilon'}^N}\tilde{\mathcal{K}}_{N}^{\epsilon,\epsilon'}f\left(\tilde{\bm{H}}^N(s),\tilde{\bm{\Theta}}^N(s)\right)\de s
\end{equation*}
and  the process 
\begin{equation*}
\bm{Z}^{N,f}(t):= f\left(\tilde{\bm{H}}^N(t\wedge\uptau_{-\frac{1}{\epsilon},-\epsilon'}^N)\right)-f\left(\tilde{\bm{H}}^N(0)\right)-\int_0^{t\wedge\uptau_{-\frac{1}{\epsilon},-\epsilon'}^N}\mathcal{A}^{\epsilon,\epsilon'}f\left(\tilde{\bm{H}}^N(s),\tilde{\bm{\Theta}}^N(s)\right)\de s
\end{equation*}
with $\mathcal{A}^{\epsilon,\epsilon'}:=\mathbbm{1}_{(-\frac{1}{\epsilon},-\epsilon')}\mathcal{A} $  (see formula \eqref{def: gen A}).
Observing that $\tilde{\mathcal{K}}_{N}^{\epsilon,\epsilon'}f(h,\theta)= \mathcal{A}^{\epsilon,\epsilon'}f(h,\theta) + o(1)$, we have $\forall\ m\geq1$,  $\forall\  g_1,\ldots,g_m$ continuous and bounded functions on $[-\frac{1}{\epsilon},-\epsilon']$ and $0\leq t_1\leq \ldots\leq t_m\leq s\leq t\leq T$, 
\begin{equation*}
E\left[  \left(\bm{Z}^{N,f}(t)-\bm{Z}^{N,f}(s) \right) g_1\left(\tilde{\bm{H}}^N(t_1\wedge\uptau_{-\frac{1}{\epsilon},-\epsilon'}^N)\right) \cdot\ldots\cdot g_m\left(\tilde{\bm{H}}^N(t_m\wedge\uptau_{-\frac{1}{\epsilon},-\epsilon'}^N) \right)\right] = o(1).
\end{equation*} 

Let us write it explicitly 
\begin{multline}\label{eq: explicit mart property}
E\left[  \left( f\left(\tilde{\bm{H}}^N(t\wedge\uptau_{-\frac{1}{\epsilon},-\epsilon'}^N )\right) - f\left(\tilde{\bm{H}}^N(s\wedge\uptau_{-\frac{1}{\epsilon},-\epsilon'}^N) \right)\right)  g_1\left(\tilde{\bm{H}}^N(t_1\wedge\uptau_{-\frac{1}{\epsilon},-\epsilon'}^N)\right) \cdot\ldots\cdot \right. \\
\left.  g_m\left(\tilde{\bm{H}}^N(t_m\wedge\uptau_{-\frac{1}{\epsilon},-\epsilon'}^N) \right) \right]   +  \\
E\Bigg[ \Bigg( - \int_{s\wedge\uptau_{-\frac{1}{\epsilon},-\epsilon'}^N}^{t\wedge\uptau_{-\frac{1}{\epsilon},-\epsilon'}^N}\mathcal{A}^{\epsilon,\epsilon'}f\left(\tilde{\bm{H}}^N(u),\tilde{\bm{\Theta}}^N(u)\right)\de u\Bigg)  g_1\left(\tilde{\bm{H}}^N(t_1\wedge\uptau_{-\frac{1}{\epsilon},-\epsilon'}^N)\right) \cdot\ldots\cdot \\
 g_m\left(\tilde{\bm{H}}^N(t_m\wedge\uptau_{-\frac{1}{\epsilon},-\epsilon'}^N) \right)\Bigg]  = o(1).
\end{multline}
Note that all the terms in the expectations above are uniformly bounded with respect to $N$.  Consider the second term of \eqref{eq: explicit mart property} and observe that:  
\begin{eqnarray*}
\int_{s\wedge\uptau_{-\frac{1}{\epsilon},-\epsilon'}^N}^{t\wedge\uptau_{-\frac{1}{\epsilon},-\epsilon'}^N}\mathcal{A}^{\epsilon,\epsilon'}f\left(\tilde{\bm{H}}^N(u),\tilde{\bm{\Theta}}^N(u)\right)\de u &=& \int_{s}^{t}\mathcal{A}^{\epsilon,\epsilon'}f\left(\tilde{\bm{H}}^N(u\wedge\uptau_{-\frac{1}{\epsilon},-\epsilon'}^N),\tilde{\bm{\Theta}}^N(u\wedge\uptau_{-\frac{1}{\epsilon},-\epsilon'}^N )\right)\de u \\
&=&  \int_{s}^{t}\mathcal{A}f\left(\tilde{\bm{H}}^N(u\wedge\uptau_{-\frac{1}{\epsilon},-\epsilon'}^N),\tilde{\bm{\Theta}}^N(u\wedge\uptau_{-\frac{1}{\epsilon},-\epsilon'}^N )\right)\de u
\end{eqnarray*}
where the last equality comes from the fact that $f$ has compact support. Therefore, the conclusion follows from \eqref{eq: explicit mart property} using \eqref{eq: converg as averag} and dominated convergence theorem. \\

Now we observe that the process $\tilde{\bm{H}}_{-\frac{1}{\epsilon},-\epsilon'}(\cdot\wedge\tilde{\uptau}_{-\frac{1}{\epsilon}+\delta,-\epsilon'-\delta}) $ is a solution of the $\big(\bar{\mathcal{A}}, U\big)$-stopped martingale problem with $U=(-\frac{1}{\epsilon}+\delta,-\epsilon'-\delta)$ . Indeed for each given $g$ in $\mathcal{C}^3_0\big( [-\frac{1}{\epsilon}+\delta,-\epsilon'-\delta]\big)$ (which is measure determining for $\mathcal{D}\big([0,T];[-\frac{1}{\epsilon}+\delta,-\epsilon'-\delta]\big)$ ) there exists a function $f\in \mathcal{C}^3_c\big((-\frac{1}{\epsilon},-\epsilon')\big)$ such that $g(x)=f(x)$ for all $x\in(-\frac{1}{\epsilon}+\delta,-\epsilon'-\delta)$ and so
\begin{eqnarray*}
\bm{Z}^f_{\frac{1}{\epsilon},-\epsilon'}(t\wedge\tilde{\uptau}_{-\frac{1}{\epsilon}+\delta,-\epsilon'-\delta})  = && g\big(\tilde{\bm{H}}_{-\frac{1}{\epsilon},-\epsilon'}(t\wedge \tilde{\uptau}_{-\frac{1}{\epsilon}+\delta,-\epsilon'-\delta})\big)-g\big(\tilde{\bm{H}}_{-\frac{1}{\epsilon},-\epsilon'}(0)\big) \\
 && -\int_0^{t\wedge\tilde{\uptau}_{-\frac{1}{\epsilon}+\delta,-\epsilon'-\delta}}\bar{\mathcal{A}}g\big(\tilde{\bm{H}}_{-\frac{1}{\epsilon},-\epsilon'}(s)\big)\de s
\end{eqnarray*}
is a martingale. Moreover, by Proposition \ref{existenceuniqueness} the  martingale problem for $\bar{\mathcal{A}}$ is well-posed and has solution $\bm{H}$;  then,  by Theorem \ref{teo:Kurtz stopping},  $\tilde{\bm{H}}_{-\frac{1}{\epsilon},-\epsilon'}(\cdot\wedge \tilde{\uptau}_{-\frac{1}{\epsilon}+\delta,-\epsilon'-\delta})$ and $\bm{H}(\cdot\wedge\uptau_{-\frac{1}{\epsilon}+\delta,-\epsilon'-\delta})$ have the same distribution. 
\end{proof}
We are ready to prove Theorem \ref{Main}.
\paragraph{Proof of Theorem \ref{Main}}
We have to show that, for every $\epsilon>0$,  $\tilde{\bm{H}}^N(\cdot\wedge\uptau_\epsilon)$ converges to $\bm{H}(\cdot\wedge\uptau_\epsilon)$ as $N\to\infty$.
First of all observe that the weak limit $\tilde{\bm{H}}_{-\frac{1}{\epsilon},-\epsilon'}$ of the sequence $\{\tilde{\bm{H}}^N(\cdot\wedge\uptau_{-\frac{1}{\epsilon},-\epsilon'}^N); N>1\}$ is continuous a.s., hence convergence also holds endowing the space $\mathcal{D}([0,T],\mathbb{R})$ with the uniform topology (see \cite{Silvestrov}). Let $f\in C_b(\mathcal{D}([0,T],\mathbb{R})$ and consider
\begin{equation*}
S_\epsilon= \left|E\left[f\left(\tilde{\bm{H}}^N(\cdot\wedge\uptau_{-\frac{1}{\epsilon},0}^N)\right)\right]-E\left[f\left(\bm{H}(\cdot\wedge\uptau_{-\frac{1}{\epsilon},0})\right)\right]\right|.
\end{equation*}

For any $N>1$ and $\delta,\epsilon'>0$ we write
\begin{align}
S_\epsilon\leq &  \left|E\left[f\left(\tilde{\bm{H}}^N(\cdot\wedge\uptau_{-\frac{1}{\epsilon},0}^N)\right)\right] -E\left[f\left(\tilde{\bm{H}}^N(\cdot\wedge\uptau_{-\frac{1}{\epsilon},-\epsilon'}^N)\right)\right]\right| \nonumber \\
+ & \left|E\left[f\left(\tilde{\bm{H}}^N(\cdot\wedge\uptau_{-\frac{1}{\epsilon},-\epsilon'}^N)\right)\right] - E\left[f\left(\tilde{\bm{H}}_{-\frac{1}{\epsilon},-\epsilon'}\right)\right]\right| \nonumber\\ 
+ & \left|E\left[f\left(\tilde{\bm{H}}_{-\frac{1}{\epsilon},-\epsilon'}\right)\right] - E\left[f\left(\tilde{\bm{H}}_{-\frac{1}{\epsilon},-\epsilon'}(\cdot\wedge\tilde{\uptau}_{-\frac{1}{\epsilon}+\delta,-\epsilon'-\delta})\right)\right] \right|  \label{localization}\\
+ &\left|E\left[f\left(\tilde{\bm{H}}_{-\frac{1}{\epsilon},-\epsilon'}(\cdot\wedge\tilde{\uptau}_{-\frac{1}{\epsilon}+\delta,-\epsilon'-\delta})\right)\right] -E\left[f\left(\bm{H}(\cdot\wedge\uptau_{-\frac{1}{\epsilon}+\delta,-\epsilon'-\delta})\right)\right]\right|  \nonumber\\
+ &\left|E\left[f\left(\bm{H}(\cdot\wedge\uptau_{-\frac{1}{\epsilon}+\delta,-\epsilon'-\delta})\right)\right]-E\left[f\left(\bm{H}(\cdot\wedge\uptau_{-\frac{1}{\epsilon}-\delta,0})\right)\right]\right| \nonumber\\
+ & \left| E\left[f\left(\bm{H}(\cdot\wedge\uptau_{-\frac{1}{\epsilon}-\delta,0})\right)\right] - E\left[f\left(\bm{H}(\cdot\wedge\uptau_{-\frac{1}{\epsilon},0})\right)\right] \right|. \nonumber
\end{align}
We first estimate the quantities related to the macroscopic process. 
By Proposition \ref{prop same distribution} 
\begin{equation*}
\left|E\left[f\left(\tilde{\bm{H}}_{-\frac{1}{\epsilon},-\epsilon'}(\cdot\wedge\tilde{\uptau}_{-\frac{1}{\epsilon}+\delta,-\epsilon'-\delta})\right)\right] -E\left[f\left(\bm{H}(\cdot\wedge\uptau_{-\frac{1}{\epsilon}+\delta,-\epsilon'-\delta})\right)\right]\right|  = 0.
\end{equation*}
Now, let us fix $\gamma>0$. 
The processes $\bm{H}$ and $\tilde{\bm{H}}_{-\frac{1}{\epsilon},-\epsilon'}$ are continuous and, as $\delta\to 0$, we have $\tilde{\uptau}_{-\frac{1}{\epsilon}+\delta,-\epsilon'-\delta}\longrightarrow\tilde{\uptau}_{-\frac{1}{\epsilon},-\epsilon'}$ and $\uptau_{-\frac{1}{\epsilon}-\delta,0}\longrightarrow\uptau_{-\frac{1}{\epsilon},0}$. 
Then, we can choose $\delta$ small enough such that
\begin{eqnarray}
&&\left| E\left[f\left(\tilde{\bm{H}}_{-\frac{1}{\epsilon},-\epsilon'}\right)\right] - E\left[f\left(\tilde{\bm{H}}_{-\frac{1}{\epsilon},-\epsilon'}(\cdot\wedge\tilde{\uptau}_{-\frac{1}{\epsilon}+\delta,-\epsilon'-\delta})\right)\right] \right| < \gamma, \nonumber \\
&&\left| E\left[f\left(\bm{H}(\cdot\wedge\uptau_{-\frac{1}{\epsilon}-\delta,0})\right)\right] - E\left[f\left(\bm{H}(\cdot\wedge\uptau_{-\frac{1}{\epsilon},0})\right)\right] \right| < \gamma. \nonumber
\end{eqnarray} 
Define the exit times of $\bm{H}$ from the left and right  boundaries of the domain as
\begin{eqnarray}\label{def: leftexit}
l_{-\frac{1}{\epsilon}+\delta,-\epsilon'-\delta}:= \inf\left\{ t\in[0,T]: \bm{H}(t)\leq -\tfrac{1}{\epsilon}+\delta,\  \bm{H}(s)\in\left(-\tfrac{1}{\epsilon}+\delta,-\epsilon'-\delta\right)\ \forall\ s<t\right\};\\	\label{def: rightexit}
r_{-\frac{1}{\epsilon}+\delta,-\epsilon'-\delta}:= \inf\left\{ t\in[0,T]: \bm{H}(t)\geq -\epsilon'-\delta,\  \bm{H}(s)\in\left(-\tfrac{1}{\epsilon}+\delta,-\epsilon'-\delta\right)\ \forall\ s<t\right\}. 
\end{eqnarray}
 By Proposition \ref{existenceuniqueness}, we can choose $-\epsilon'$ small enough such that:
\begin{equation*}
\left|E\left[f\left(\bm{H}(\cdot\wedge\uptau_{-\frac{1}{\epsilon}+\delta,-\epsilon'-\delta})\right)\right]-E\left[f\left(\bm{H}(\cdot\wedge\uptau_{-\frac{1}{\epsilon}-\delta,0})\right)\right]\right| \leq \|f\|_{\infty}P\left(r_{-\frac{1}{\epsilon}+\delta,-\epsilon'-\delta}<l_{-\frac{1}{\epsilon}+\delta,-\epsilon'-\delta})\right)<\gamma. 
\end{equation*}
We are left to estimate the first two terms of inequality \eqref{localization}.  
Let $r_{-\frac{1}{\epsilon}+\delta,-\epsilon'-\delta}^N$ and $l_{-\frac{1}{\epsilon}+\delta,-\epsilon'-\delta}^N$ be as in \eqref{def: leftexit} and \eqref{def: rightexit} with $\tilde{\bm{H}}^N$ in place of $\bm{H}$. Analogously, 
we define $ \tilde{l}_{-\frac{1}{\epsilon}-\sfrac{\delta}{2},-\epsilon'-\delta}$ and  $\tilde{r}_{-\frac{1}{\epsilon}-\sfrac{\delta}{2},-\epsilon'-\delta}$ taking $\tilde{\bm{H}}_{-\frac{1}{\epsilon}-\delta,-\epsilon'}$ in place of $\bm{H}$.
For the first term we have 
\begin{align*}
 \left|E\left[f\left(\tilde{\bm{H}}^N(\cdot\wedge\uptau_{-\frac{1}{\epsilon},0}^N)\right)\right] -E\left[f\left(\tilde{\bm{H}}^N(\cdot\wedge\uptau_{-\frac{1}{\epsilon},-\epsilon'}^N)\right)\right]\right| \leq & 
 \|f\|_{\infty}P\left(r_{-\frac{1}{\epsilon},-\epsilon'}^N<l_{-\frac{1}{\epsilon},-\epsilon'}^N)\right)  \\
 \leq &  \|f\|_{\infty}P\left(r_{-\frac{1}{\epsilon},-\epsilon'-\delta}^N<l_{-\frac{1}{\epsilon},-\epsilon'-\delta}^N)\right) \\
 \leq & \|f\|_{\infty}P\left(r_{-\frac{1}{\epsilon}-\sfrac{\delta}{4},-\epsilon'-\delta}^N<l_{-\frac{1}{\epsilon}-\frac{\delta}{4},-\epsilon'-\delta}^N)\right). 
\end{align*}
 Consider the closed set 
\begin{equation*}
R_{\delta,\frac{\delta}{4}}:=\left\{ \bm{x}\in \mathcal{D}\left([0,T],\mathbb{R}\right) : \exists \bar{t}\in [0,T]\ s.t.\ \bm{x}(\bar{t})\geq -\epsilon'-\delta\ \text{and}\   -\tfrac{1}{\epsilon}+\tfrac{\delta}{4}\leq \bm{x}(s) <-\epsilon'-\delta,\ \forall s<\bar{t}   \right\}
\end{equation*}
and observe that
\begin{equation*}
P\left(r_{-\frac{1}{\epsilon}-\sfrac{\delta}{4},-\epsilon'-\delta}^N<l_{-\frac{1}{\epsilon}-\sfrac{\delta}{4},-\epsilon'-\delta}^N)\right) \leq  P\left(\tilde{\bm{H}}^N(\cdot\wedge\uptau_{-\frac{1}{\epsilon}-\delta,-\epsilon'}^N)\in R_{\delta,\frac{\delta}{4}} \right); 
\end{equation*}
then by Portmanteau Theorem it follows   
\begin{equation*}
\limsup_N P\left(\tilde{\bm{H}}^N(\cdot\wedge\uptau_{-\frac{1}{\epsilon}-\delta,-\epsilon'}^N)\in R_{\delta,\frac{\delta}{4}} \right)\leq P\left(\tilde{\bm{H}}_{-\frac{1}{\epsilon}-\delta,-\epsilon'}\in R_{\delta,\frac{\delta}{4}} \right).
\end{equation*}
Since the processes $\tilde{\bm{H}}_{-\frac{1}{\epsilon}+\delta,-\epsilon'}(\cdot\wedge\tilde{\uptau}_{-\frac{1}{\epsilon}+\sfrac{\delta}{2},-\epsilon'+\delta})$ and $\bm{H}(\cdot\wedge\uptau_{-\frac{1}{\epsilon}+\sfrac{\delta}{2},-\epsilon'+\delta})$ have the same distribution, we obtain
\begin{equation*}
P\left(\tilde{\bm{H}}_{-\frac{1}{\epsilon}-\delta,-\epsilon'}\in R_{\delta,\frac{\delta}{4}}\right) \leq P\left( \tilde{r}_{-\frac{1}{\epsilon}-\sfrac{\delta}{2},-\epsilon'+\delta}<  \tilde{l}_{-\frac{1}{\epsilon}-\sfrac{\delta}{2},-\epsilon'+\delta} \right) = P\left( r_{-\frac{1}{\epsilon}-\sfrac{\delta}{2},-\epsilon'+\delta}< l_{-\frac{1}{\epsilon}-\sfrac{\delta}{2},-\epsilon'+\delta} \right).
\end{equation*}
 Using again Proposition \ref{existenceuniqueness}, we can choose $-\epsilon'$ small enough and $N$ big enough such that 
\begin{equation*}
 \left|E\left[f\left(\tilde{\bm{H}}^N(\cdot\wedge\uptau_{-\frac{1}{\epsilon},0}^N)\right)\right] -E\left[f\left(\tilde{\bm{H}}^N(\cdot\wedge\uptau_{-\frac{1}{\epsilon},-\epsilon'}^N)\right)\right]\right| <\gamma.
\end{equation*} 
Finally, for the second term of \eqref{localization}, by the convergence of $\tilde{\bm{H}}^N(\cdot\wedge\uptau_{-\frac{1}{\epsilon},-\epsilon'}^N)$ to $\tilde{\bm{H}}_{-\frac{1}{\epsilon},-\epsilon'}$, we can take $N$ big enough such that
\begin{equation*}
\left|E\left[f\left(\tilde{\bm{H}}^N(\cdot\wedge\uptau_{-\frac{1}{\epsilon},-\epsilon'}^N)\right)\right] - E\left[f\left(\tilde{\bm{H}}_{-\frac{1}{\epsilon},-\epsilon'}\right)\right]\right| <\gamma 
\end{equation*}
and the proof is complete.\\

\vspace{1cm}
\noindent {\bf Acknowledgement.} \\
\noindent 
{
The authors thank Paolo Dai Pra for useful discussions and suggestions. The authors are also grateful to Fabio Antonelli for suggesting them some references. 
The authors are members of the Gruppo Nazionale per
l'Analisi Matematica, la Probabilit\`a e le loro Applicazioni (GNAMPA) of the Istituto Nazionale di Alta Matematica (INdAM).
}

\bibliography{bibliografy}
\bibliographystyle{amsplain}

\end{document}